\documentclass[11pt]{amsart}
\usepackage[english]{babel}
\usepackage[T1]{fontenc}
\usepackage[latin1]{inputenc}
\usepackage{graphicx}
\usepackage{amsmath,amssymb,amsthm,mathrsfs,,txfonts,ifthen,xr}
\usepackage{soul}
\usepackage{array, multirow}
\usepackage{stmaryrd}
\usepackage{mathrsfs}
\usepackage{color}
\usepackage{hyperref}
\usepackage{enumerate}
\usepackage[all]{xy}
\usepackage{tikz-cd}
\usetikzlibrary{decorations.pathreplacing}

\usepackage{ytableau}

\theoremstyle{plain}
\newtheorem{theo}{Theorem}[section]
\newtheorem{lemma}[theo]{Lemma}
\newtheorem{proposition}[theo]{Proposition}
\newtheorem{corollary}[theo]{Corollary}

\theoremstyle{definition}
\newtheorem{definition}[theo]{Definition}

\newtheorem{notation}[theo]{Notation} 

\newtheorem{example}[theo]{Example}

\theoremstyle{remark}
\newtheorem{remark}[theo]{Remark}

\def\A{{\rm A}}
\def\B{{\rm B}}
\def\C{{\rm C}}
\def\D{{\rm D}}

\def\H{{\rm H}}
\def\I{{\rm I}}
\def\J{{\rm J}}
\def\K{{\rm K}}
\def\L{{\rm L}}
\def\M{{\rm M}}
\def\N{{\rm N}}
\def\O{{\rm O}}
\def\P{{\rm P}}

\def\Q{{\rm Q}}

\def\R{{\rm R}}
\def\S{{\rm S}}

\def\T{{\rm T}}
\def\U{{\rm U}}
\def\V{{\rm V}}

\def\X{{\rm X}}
\def\Y{{\rm Y}}
\def\Z{{\rm Z}}

\def\tr{{\rm tr}}
\def\pr{{\rm pr}}
\def\Sp{{\rm Sp}}
\def\GL{{\rm GL}}	
\def\Id{{\rm Id}}
\def\Re{{\rm Re}}
\def\Im{{\rm Im}}
\def\End{{\rm End}}
\def\Lie{{\rm Lie}}
\def\dim{{\rm dim}}
\def\Mat{{\rm Mat}}

\def\Rad{{\rm Rad}}
\def\trd{{\rm trd}}
\def\log{{\rm log}}
\def\BCH{{\rm BCH}}
\def\Spec{{\rm Spec}}
\def\diag{{\rm diag}}
\def\cosh{{\rm cosh}}
\def\sinh{{\rm sinh}}

\allowdisplaybreaks
\title{The Cone of J-Hermitian Matrices and a Geometric Mean}
\author{Jose Franco}
\address{Department of Mathematics and Statistics\\ University of North Florida \\ 1 UNF Drive \\ Jacksonville \\ FL 32224 \\ USA}
\email{jose.franco@unf.edu}
\author{Allan Merino}
\address{Department of Mathematics and Statistics\\ University of North Florida \\ 1 UNF Drive \\ Jacksonville \\ FL 32224 \\ USA}
\email{allan.merino@unf.edu}



\date{}

\begin{document}

\begin{abstract}

We study the cone $\mathscr{P}_{\J}$ of positive $\J$-Hermitian matrices associated with an indefinite signature matrix $\J = \Id_{p,q}$. We show that the $\J$-exponential map is bijective and use it to analyze the algebraic and geometric structure of $\mathscr{P}_{\J}$. Through a canonical identification with the cone of positive definite matrices, we endow $\mathscr{P}_{\J}$ with a natural Riemannian structure. In this setting, we define a $\J$-geometric mean as the midpoint of geodesics and prove that it is uniquely characterized as the solution of a Riccati-type equation.

\end{abstract}

\maketitle

\tableofcontents

\section{Introduction}

The problem of defining the geometric mean of two positive definite matrices was first satisfactorily solved by Pusz and Woronowicz \cite{Pusz} by using functional calculus. They defined it by
\begin{equation*}
\A \sharp \B = \A^{\frac{1}{2}}\left(\A^{-\frac{1}{2}}\B\A^{-\frac{1}{2}}\right)^{\frac{1}{2}}\A^{\frac{1}{2}}\,.
\end{equation*}
Then from the systematic study of operator means by Kubo and Ando \cite{KuboAndo1980} the weighted version of this mean can naturally be defined for $t \in \left[0\,, 1\right]$ by
\begin{equation*}
\A \sharp_{t} \B = \A^{\frac{1}{2}}\left(\A^{-\frac{1}{2}}\B\A^{-\frac{1}{2}}\right)^{t}\A^{\frac{1}{2}}\,.
\end{equation*}
From that point on, this mean has been studied from many different points of view. Take for example the Ando-Hiai inequality \cite{AndoHiai1994} that states that
\begin{equation*}
\A \sharp_{t} \B \leq \Id \qquad \Longrightarrow \qquad \A^{p} \sharp_{t} \B^{p} \leq \Id\,, \qquad \left(p \geq 1\right)\,.
\end{equation*}
This inequality has several applications, from the Golden-Thompson inequality, to the Furuta Inequality.

\noindent In another direction, the weighted geometric means of positive definite matrices have been studied from the point of view of Riemannian geometry. In \cite{BHATIAH}, Bhatia and Holbrook establish the weighted geometric means as the geodesic curves connecting two points in the cone of positive definite matrices. Closely related to this viewpoint is the fact that CAT(0) spaces admit unique barycenters. In this sense, the geometric mean is realized as the barycenter 
\begin{equation*}
\A \sharp \B = \arg\min_{\X \geq 0}\left(\delta_{2}(\X\,, \A)^{2} + \delta_{2}(\X\,, \B)^{2}\right)\,.
\end{equation*}
This perspective leads naturally to multivariate geometric means, but this topic is not considered in this article. The interested reader is referred to \cite{LawsonLim2014} for more details\,.

\noindent Then there is the problem of extending the definition of the geometric means to other spaces beyond positive definite matrices. For example, Hiai and Kosaki extended the definition to positive $\tau$-measurable operators and positive elements in Haagerup's $\L^p$ spaces \cite{HiaiKosaki2021}. In the recent paper of Choi, Kim, and Lim \cite{ChoiKimLim2022}, they provide a binomial formula for the weighted geometric means of unipotent matrices. At this point, we reference the work in \cite{TRIO}, where they established a geometric mean for symmetric spaces of noncompact type through a Lie theoretic approach.  

\noindent Functional calculus is a key ingredient in these constructions, but it relies on the spectral theorem and hence on the self-adjointness of the operators. When operators are not Hermitian (or even normal), functional calculus is no longer available and the classical theory of means of positive operators is no longer viable. In many applications, however, this limitation is not restrictive, since operators are typically self-adjoint.

\noindent In relativistic or indefinite-metric settings, the relevant operators are often not self-adjoint in the usual Hilbert space sense. For instance, in Minkowski space, Lorentz-Hermitian operators play a role analogous to Hermitian operators in quantum theory. This motivates the development of alternative frameworks extending the theory of means beyond the positive definite case to the Krein spaces, a problem we address here for the geometric mean. This question has been studied in \cite{Dehghani}. However, they were unsuccessful in generalizing the theory to this setting via operator theoretical techniques. The failure of operator-theoretic techniques in the indefinite setting motivates a structural approach based on Lie-theoretic and representation-theoretic tools, which allows us to obtain explicit formulas and global geometric descriptions unavailable through classical functional calculus. The study of matrix analysis topics on Krein spaces is not new, see for example \cite{BebianoLemosProvJrSoares2012, BebianoLemosSoares2023,  BebianoLemosSoares2025} to name a few.

\medskip

Let $\mathbb{D} \in \left\{\mathbb{R}\,, \mathbb{C}\,, \mathbb{H}\right\}$, and let $\J$ be the matrix in $\Mat(n) := \Mat(n\times n\,, \mathbb{D})$ given by $\J = \diag\left(\Id_{p}\,, -\Id_{q}\right)$, with $p+q=n$. Using $\J$, we define an involution $\sharp$ on $\Mat(n)$ as
\begin{equation*}
\A \mapsto \A^{\sharp} := \J\A^{*}\J\,,
\end{equation*}
with $\A^{*} = \overline{\A}^{t}$, which plays the role of an adjoint in the indefinite setting. We denote by
\begin{equation*}
\mathfrak{p}_{\J} := \left\{\X \in \Mat(n)\,, \X^{\sharp}=\X\right\}
\end{equation*}
the real vector space of $\J$-Hermitian matrices, and let
\begin{equation*}
\mathfrak{p} = \left\{\X \in \Mat(n)\,, \X = \X^{*}\right\}\,.
\end{equation*}

\noindent Let $\mathscr{P}$ be the subset of $\mathfrak{p}$ given by
\begin{equation*}
\mathscr{P} := \left\{\X \in \mathfrak{p}\,, \X \text{ is positive}\right\}\,.
\end{equation*}
It is well-known that the exponential map $\exp: \mathfrak{p} \to \mathscr{P}$ is well-defined and bijective. In this paper, we define in $\mathfrak{p}_{\J}$ a cone that is analogue to $\mathscr{P}$. However, the restriction of the exponential map to $\mathfrak{p}_{\J}$ is not injective, so this is not how we proceed in our work\,.

\noindent Within $\mathfrak{p}_{\J}$, we distinguish the open cone
\begin{equation*}
\mathscr{P}_{\J} := \left\{\A \in \mathfrak{p}_{\J}\,, \J\A \text{ is positive}\right\}\,,
\end{equation*}
which we call the cone of $\J$-positive matrices (in $\mathfrak{p}_{\J}$). Although $\mathscr{P}_{\J}$ is not stable under multiplication, it enjoys several properties analogous to those of the classical cone $\mathscr{P} := \exp(\mathfrak{p})$ of positive definite Hermitian matrices. In particular, $\mathscr{P}_{\J}$ is stable under inversion, and the group $\GL(n)$ acts transitively on $\mathscr{P}_{\J}$ by congruence.

\noindent At this point, we would like to contrast our results to those obtained in \cite{TRIO}. While the space $\mathscr{P}_{\J}$ is isomorphic to a symmetric space of noncompact type, from the study in \cite{TRIO}, it is not clear how to proceed from a computational point of view. The clearest example of this is the quaternionic case, where the Riemannian metric arises from a reduced trace instead of the regular trace (see Appendix \ref{AppendixA}). Their construction relies on the classical exponential map at the Lie algebra level, whose global bijectivity fails in the indefinite setting when $\J \neq \Id_{n}$. Therefore, our work should be considered complementary to that of \cite{TRIO}\,.

\noindent A central role in this paper is played by the linear map
\begin{equation*}
\Phi_{\J}: \mathfrak{p}_{\J} \ni \X \longrightarrow \J\X \in \mathfrak{p}\,,
\end{equation*}
which restricts to a smooth bijection between $\mathscr{P}_{\J}$ and the classical positive cone $\mathscr{P}$. This identification allows us to transport both analytic and geometric structures from $\mathscr{P}$ to $\mathscr{P}_{\J}$. In particular, we define the $\J$-exponential map by
\begin{equation*}
\exp_{\J}(\X) := \J\exp(\J\X)\,, \qquad \left(\X \in \mathfrak{p}_{\J}\right)\,,
\end{equation*}
which yields a global diffeomorphism from $\mathfrak{p}_{\J}$ onto $\mathscr{P}_{\J}$. The bijectivity follows immediately from the bijectivity of the classical exponential $\exp$ on $\mathfrak{p}$ and the linear isomorphism $\Phi_{\J}$. Its inverse $\log_{\J}$ provides a natural framework to define geodesics and a Riemannian structure on $\mathscr{P}_{\J}$ by pullback from the classical geometry of $\mathscr{P}$.

\noindent On the set $\Mat(n)$, we denote by $\bullet$ the multiplication given by
\begin{equation*}
\X \bullet \Y = \X\J\Y\,, \qquad \left(\X\,, \Y \in \Mat(n)\right)\,.
\end{equation*}
This multiplication is associative and is such that $\X \bullet \J = \J \bullet \X = \J$. Moreover, a matrix $\A$ is invertible if and only if there exists a matrix $\B \in \Mat(n)$ such that $\A \bullet \B = \B \bullet \A = \J$, and we have $\B = \J\A^{-1}\J$. This multiplication allows us to define geodesics in a particularly simple form.

\noindent We denote by $\langle\cdot\,, \cdot\rangle$ the Riemannian metric on $\mathscr{P}$ given by
\begin{equation*}
\langle\U\,, \V\rangle_{\P} = \widetilde{\tr}(\P^{-1}\U\P^{-1}\V)\,, \qquad \qquad \left(\P \in \mathscr{P}\,, \U\,, \V \in \T_{\P}(\mathscr{P}) \cong \mathfrak{p}\right)\,,
\end{equation*}
where $\widetilde{\tr}: \Mat(n) \to \mathbb{R}$ is given by
\begin{equation*}
\widetilde{\tr}(\X) = \begin{cases} \tr(\X) & \text{ if } \mathbb{D} \in \left\{\mathbb{R}\,, \mathbb{C}\right\} \\ \Re(\tr(\X)) & \text{ if } \mathbb{D} = \mathbb{H} \end{cases}
\end{equation*}
and let $\omega := \Phi^{*}_{\J}\langle\cdot\,, \cdot\rangle$ be the pull-back of the Riemannian metric to $\mathscr{P}_{\J}$. In Section \ref{SectionFour}, we prove the following theorem\,.

\begin{theo}

Let $\A,\B\in \mathscr{P}_{\J}$. The map $\gamma: \left[0\,, 1\right] \to \mathscr{P}_{\J}$ given by
\begin{equation*}
\gamma(t) := \A^{\frac{1}{2}}_{\J} \bullet \left(\A^{-\frac{1}{2}}_{\J} \bullet \B \bullet \A^{-\frac{1}{2}}_{\J}\right)^{t}_{\J} \bullet \A^{\frac{1}{2}}_{\J}\,, \qquad \left(t \in \left[0\,, 1\right]\right)\,,
\end{equation*}
is the equation of the geodesic between $\A$ and $\B$ in $\mathscr{P}_{\J}$ with respect to $\omega$, where $\X^{t}_{\J}$ is the matrix in $\mathscr{P}_{\J}$ given by
\begin{equation*}
\X^{t}_{\J} = \exp_{\J}\left(t\log_{\J}(\X)\right)\,, \qquad \left(\X \in \mathscr{P}_{\J}\right)\,.
\end{equation*}

\end{theo}

\noindent Within this setting, we introduce a $\J$-geometric mean as the midpoint of geodesics in $\mathscr{P}_{\J}$, i.e.
\begin{equation*}
\A \sharp^{\J} \B := \A^{\frac{1}{2}}_{\J}\bullet\left(\A^{-\frac{1}{2}}_{\J}\bullet\B\bullet\A^{-\frac{1}{2}}_{\J}\right)^{\frac{1}{2}}_{\J}\bullet\A^{\frac{1}{2}}_{\J}\,, \qquad \left(\A\,, \B \in \mathscr{P}_{\J}\right)\,.
\end{equation*}
Our main results show that this mean admits an explicit closed formula paralleling the Pusz--Woronowicz construction and satisfies fundamental properties such as congruence invariance under $\GL(n)$, monotonicity with respect to the $\J$-Loewner order, and compatibility with the classical geometric mean under the identification $\Phi_{\J}$.

\noindent The theoretical framework is developed in Section~\ref{SectionOne} and the functional calculus is introduced in Section~\ref{SectionTwo}. In Section~\ref{SectionThree}, we introduce the $\J$-Loewner order in our case. In Section~\ref{SectionFour}, we study the geodesic curves on the cone of $\J$-positive matrices and arrive at the definition of the weighted geometric means of $\J$-positive matrices and we conclude by studying their fundamental properties in Section~\ref{SectionFive}. In Appendix~\ref{AppendixA}, we give a construction of the cone of positive quaternionic Hermitian matrices.

\section{The cone of J-Hermitian matrices}

\label{SectionOne}

Let $\mathbb{D} \in \left\{\mathbb{R}\,, \mathbb{C}\,, \mathbb{H}\right\}$, and let $\iota$ be the involution on $\mathbb{D}$ given by $\iota(\X) = \overline{\X}$; in particular, $\iota$ is trivial for $\mathbb{D} = \mathbb{R}$, and we have for $\mathbb{D} = \mathbb{H}$ (see Appendix \ref{AppendixA1}) that
\begin{equation*}
\iota(\X := a + ib + jc + kd) = a - ib - jc - kd\,, \qquad \left(\X \in \mathbb{H}\right)\,.
\end{equation*}

\noindent Let $\V$ be a right-$\mathbb{D}$-module with $\dim_{\mathbb{D}}(\V) = n$, and let $\B$ be a non-degenerate $\left(\iota\,, 1\right)$-hermitian form on $\V$, i.e.  for all $x\,, y\,, z \in \V$ and all $\alpha\,, \beta \in \mathbb{D}$
\begin{enumerate}
\item $\B(x\,, y\alpha + z\beta) = \B(x\,, y)\alpha + \B(x\,, z)\beta$\,,
\item $\B(x\alpha + y\beta\,, z) = \overline{\alpha}\B(x\,, z) + \overline{\beta}\B(y\,, z)$\,,
\item $\B(y\,, x) = \overline{\B(x\,, y)}$\,,
\item $\Rad(\B) := \left\{y\in \V\,,\ \B(x\,,y) = 0\,, \thinspace (\forall x \in \V)\right\} = \left\{0\right\}$\,.
\end{enumerate}

\noindent We denote by $\End(\V) := \End_{\mathbb{D}}(\V)$ the set of linear maps on $\V$, i.e.
\begin{equation*}
\End(\V) := \left\{\T: \V \to \V\,, \T(v_{1} + v_{2}\lambda) = \T(v_{1}) + \T(v_{2})\lambda\,, v_{1}\,, v_{2} \in \V\,, \lambda \in \mathbb{D}\right\}\,,
\end{equation*}
and by $\GL(\V)$ the set of invertible maps in $\End(\V)$\,. We denote by $\Mat(n) := \Mat(n \times n\,, \mathbb{D})$ the set of $n$ by $n$ matrices with entries in $\mathbb{D}$. For all $\T \in \End(\V)$ and any right-$\mathbb{D}$-basis $\mathscr{B}_{\V} = \left\{v_{1}\,, \ldots\,, v_{n}\right\}$ of $\V$, we denote by $\Mat_{\mathscr{B}_{\V}}(\T) := \left(t_{i, j}\right)_{1 \leq i, j \leq n}$ the corresponding matrix of $\Mat(n)$ given by 
\begin{equation*}
\T(v_{i}) = \sum\limits_{j = 1}^{n} t_{j,i}v_{j}\,, \qquad \left(1 \leq i\,, j \leq n\right)\,,
\end{equation*}
and satisfying $\left[\T(v)\right] = \A\left[v\right]$, where $\left[v\right] \in \mathbb{D}^{n}$ are the coordinates of $v$ in $\mathscr{B}_{\V}$. It is well-known that the map
\begin{equation}
\Gamma: \End(\V) \ni \T \mapsto \Mat_{\mathscr{B}_{\V}}(\T) \in \Mat(n)
\label{IsomorphismGamma}
\end{equation}
is bijective, and let $\GL(n) := \Gamma(\GL(\V))$\,.

\begin{remark}

\begin{enumerate}
\item The $\mathbb{R}$-linear involution $\iota$ can be extended to $\mathbb{D}^{n}$ and $\Mat(n)$ as follows: for $v = \left(v_{i}\right)_{1 \leq i \leq n} \in \mathbb{D}^{n}$ and $\A = \left(a_{i,j}\right)_{1 \leq i, j \leq n} \in \Mat(n)$, we define 
\begin{equation*}
\iota(v) = \left(\iota(v_{i})\right)_{1 \leq i \leq n}\,, \qquad \qquad \iota(\A) = \left(\iota(a_{i,j})\right)_{1 \leq i, j \leq n}\,.
\end{equation*}
Moreover, we have $\iota(\A v) = \iota(\A)\iota(v)$\,.
\item Let $\mathscr{B}_{\V} = \left\{v_{1}\,, \ldots\,, v_{n}\right\}$ be a right-$\mathbb{D}$-basis of $\V$ and let $\Omega$ be the matrix in $\Mat(n)$ given by $\Omega = \left(\B(v_{i}\,, v_{j})\right)_{1 \leq i\,, j \leq n}$. Using that the form $\B$ is hermitian, we get that
\begin{equation*}
\B(x\,, y) = \iota(x)^{t}\Omega y\,, \qquad \qquad \left(x\,, y \in \V\right)\,.
\end{equation*}
\end{enumerate}

\label{RemarkInvolution}

\end{remark}

\begin{notation}

For a matrix $\A \in \Mat(n)$ and a vector $v \in \mathbb{D}^{n}$, we denote by $\A^{*}$ and $v^{*}$ the elements in $\Mat(n)$ and $\mathbb{D}^{n}$ given by $\A^{*} = \iota(\A)^{t}$ and $v^{*} = \iota(v)^{t}$\,.

\end{notation}

\noindent We now recall the following classical theorem\,.

\begin{theo}[Sylvester's law of inertia]

Let $\V$ be a right $\mathbb{D}$-module of dimension $n$, and let
$\B$ be a non-degenerate $\left(\iota\,, 1\right)$-hermitian form on $\V$. Then there exist integers $p\,, q \geq 0$ with $p+q = n$ and a right $\mathbb{D}$-basis $\mathscr{B} = \left\{e_{1}\,, \ldots\,, e_{n}\right\}$ of $\V$ such that the  matrix $\Mat_{\mathscr{B}}\left(\B\right)$ of $\B$ in $\mathscr{B}$ is $\Id_{p\,, q} = \diag\left(\Id_{p}\,, -\Id_{q}\right)$\,.

\noindent Equivalently,
\begin{equation*}
\B(x\,, y) =\sum\limits_{k=1}^{p}\iota(x_{k})y_{k} - \sum_{k=p+1}^{p+q}\iota(x_{k})y_{k}\,, \qquad \left(x= \sum\limits_{k=1}^{n} e_{k} x_{k}\,, y = \sum\limits_{k=1}^{n} e_{k}y_{k}\right)\,,
\end{equation*}
where $x_{k}\,, y_{k} \in\mathbb{D}$\,. 

\label{TheoremClassificationHermitian}

\end{theo}

\begin{proof}

A proof can be found in \cite{SCHARLAU}\,.

\end{proof}

\begin{remark}

Let $\T \in \End(\V)$. Using that the form $\B$ is non-degenerate, there exists a unique element $\T^{\diamond} \in \End(\V)$ such that
\begin{equation}
\B(\T(v_{1})\,, v_{2}) = \B(v_{1}\,, \T^{\diamond}(v_{2}))\,, \qquad \left(v_{1}\,, v_{2} \in \V\right)\,.
\label{TDiamond}
\end{equation}
The element $\T^{\diamond} \in \End(\V)$ is known as the adjoint of $\T$ with respect to $\B$. Moreover, the corresponding map
\begin{equation*}
\End(\V) \ni \T \to \T^{\diamond} \in \End(\V)
\end{equation*}
defines an $\mathbb{R}$-linear involution on $\End(\V)$.

\end{remark}

\noindent Let $\mathscr{B}$ be a basis of $\V$ such that $\Mat_{\mathscr{B}}(\B) = \Id_{p\,, q}$. To simplify the notations, we denote by $\J$ the matrix $\Id_{p\,, q}$, i.e.
\begin{equation*}
\J = \begin{bmatrix} \Id_{p} & 0 \\ 0 & -\Id_{q} \end{bmatrix}\,.
\end{equation*}
We denote by $\U_{\J}$ the group of isometries of $\left(\V\,, \B\right)$, i.e. the real Lie group given by
\begin{equation*}
\U_{\J} := \left\{g \in \GL(\V)\,, \B(g(v_{1})\,, g(v_{2})) = \B(v_{1}\,, v_{2})\,, \left(v_{1}\,, v_{2} \in \V\right)\right\}\,.
\end{equation*}
We denote by $\mathfrak{u}_{\J}$ the Lie algebra of $\U_{\J}$. In particular, we get
\begin{equation*}
\mathfrak{u}_{\J} = \left\{\X \in \End(\V)\,, \B(\X(v_{1})\,, v_{2}) + \B(v_{1}\,, \X(v_{2})) = 0\,, \left(\forall v_{1}\,, v_{2} \in \V\right)\right\}\,.
\end{equation*}
Similarly, we denote by $\mathfrak{p}_{\J}$ the subset of $\End(\V)$ given by
\begin{equation*}
\mathfrak{p}_{\J} := \left\{\X \in \End(\V)\,, \B(\X(v_{1})\,, v_{2}) - \B(v_{1}\,, \X(v_{2})) = 0\,, \left(\forall v_{1}\,, v_{2} \in \V\right)\right\}\,.
\end{equation*}
We have
\begin{equation*}
\End(\V) = \mathfrak{u}_{\J} \oplus \mathfrak{p}_{\J}\,,
\end{equation*}
and if we denote by $\left[\cdot\,, \cdot\right]$ the Lie bracket on $\End(\V)$ (i.e. $\left[\X\,, \Y\right] = \X \circ \Y - \Y \circ \X$ for all $\X\,, \Y \in \End(\V)$), one can see that 
\begin{equation*}
\left[\mathfrak{u}_{\J}\,, \mathfrak{u}_{\J}\right] \subseteq \mathfrak{u}_{\J}\,, \qquad \left[\mathfrak{u}_{\J}\,, \mathfrak{p}_{\J}\right] \subseteq \mathfrak{p}_{\J}\,, \qquad \left[\mathfrak{p}_{\J}\,, \mathfrak{p}_{\J}\right] \subseteq \mathfrak{u}_{\J}\,.
\end{equation*}

\begin{remark}

\begin{enumerate}
\item We identify $\V$ with $\mathbb{D}^{n}$ by replacing $v \in \V$ with its coordinates in the basis $\mathscr{B}$. Therefore, the form $\B$ on $\V$ gives rise to a non-degenerate $\left(\iota\,, 1\right)$-hermitian form $\B$ on $\mathbb{D}^{n}$\,.
\item Let $\mathscr{B}$ be a basis of $\V$ such that $\Mat_{\mathscr{B}}(\B) = \J$. Using the isomorphism \eqref{IsomorphismGamma}, we can write $\mathfrak{u}_{\J}$ and $\mathfrak{p}_{\J}$ as subsets of $\Mat(n)$. Indeed, using that $\B(u\,, v) = u^{*}\J v$ for all $u\,, v \in  \V$, we get 
\begin{eqnarray*}
\mathfrak{u}_{\J} & = & \left\{\X \in \Mat(n)\,, \left(\X v_{1}\right)^{*}\J v_{2} + v^{*}_{1}\J\left(\X v_{2}\right) = 0\,, \left(v_{1}\,, v_{2} \in \mathbb{D}^{n}\right)\right\} \\ & = & \left\{\X \in \Mat(n)\,, v^{*}_{1}\left(\X^{*}\J + \J\X\right)v_{2} = 0\,, \left(v_{1}\,, v_{2} \in \mathbb{D}^{n}\right)\right\} \\ 
& = & \left\{\X \in \Mat(n)\,, \X^{*}\J = -\J\X\right\}\,.
\end{eqnarray*}
Similarly, we have 
\begin{equation*}
\mathfrak{p}_{\J} = \left\{\X \in \Mat(n)\,, \X^{*}\J = \J\X\right\}\,.
\end{equation*}
The set $\mathfrak{p}_{\J}$ is known as the set of $\J$-hermitian matrices. Moreover, we have
\begin{equation*}
\U_{\J} = \left\{g \in \GL(n)\,, g^{*}\J g = \J\right\}\,.
\end{equation*}
\item For all $\X \in \Mat(n)$, we denote by $\X^{\sharp}$ the matrix in $\Mat(n)$ given by 
\begin{equation*}
\X^{\sharp} = \J\X^{*}\J\,.
\end{equation*}
One can easily see that for all $\X\,, \Y \in \Mat(n)$, we have $\left(\X^{\sharp}\right)^{\sharp} = \X$ and $\left(\X\Y\right)^{\sharp} = \Y^{\sharp}\X^{\sharp}$. Therefore, the map
\begin{equation*}
\Mat(n) \ni \X \to \X^{\sharp} \in \Mat(n)
\end{equation*}
is an $\mathbb{R}$-linear involution. Moreover, one can see that for all $\T \in \End(\V)$, we have
\begin{equation*}
\Mat_{\mathscr{B}}(\T^{\diamond}) = \Mat_{\mathscr{B}}(\T)^{\sharp}\,,
\end{equation*}
where $\T^{\diamond}$ was defined in Equation \eqref{TDiamond}\,.
\item We can rewrite the subsets $\mathfrak{u}_{\J}\,, \mathfrak{p}_{\J}\,,$ and $\U_{\J}$ of $\Mat(n)$ using the involution $\sharp$ defined above. Indeed, we have
\begin{equation*}
\mathfrak{u}_{\J} = \left\{\X \in \Mat(n)\,, \X^{\sharp} = -\X\right\}\,, \qquad \mathfrak{p}_{\J} = \left\{\X \in \Mat(n)\,, \X^{\sharp} = \X\right\}\,,
\end{equation*}
and 
\begin{equation*}
\U_{\J} = \left\{g \in \GL(n)\,, g^{\sharp}g = \Id_{n}\right\}\,.
\end{equation*}
\item Let $\B_{\J}$ be the form on $\mathbb{D}^{n}$ given by
\begin{equation*}
\B_{\J}(u_{1}\,, u_{2}) = \B(\J u_{1}\,, u_{2})\,, \qquad \left(u_{1}\,, u_{2} \in \mathbb{D}^{n}\right)\,.
\end{equation*}
The form $\B_{\J}$ is $(\iota\,, 1)$-hermitian and positive. Indeed, for all $u \in \mathbb{D}^{n}$, we have
\begin{equation*}
\B_{\J}(u\,, u) = \B(\J u\,, u) = \left(\J u\right)^{*}\J u = u^{*} \J^{2} u = u^{*}u = \sum\limits_{i = 1}^{n} \left|u_{i}\right|^{2}\,,
\end{equation*}
i.e. $\B_{\J}(u\,, u) > 0$ for all non-zero vector $u \in \mathbb{D}^{n}$\,.
\item Let $\mathfrak{u}\,, \mathfrak{p}$ be the subsets of $\Mat(n)$ respectively given by
\begin{equation*}
\mathfrak{u} = \left\{\X \in \Mat(n)\,, \B_{\J}(\X u_{1}\,, u_{2}) + \B_{\J}(u_{1}\,, \X u_{2}) = 0\,, \left(u_{1}\,, u_{2} \in \mathbb{D}^{n}\right)\right\}
\end{equation*}
and 
\begin{equation*}
\mathfrak{p} = \left\{\X \in \Mat(n)\,, \B_{\J}(\X u_{1}\,, u_{2}) - \B_{\J}(u_{1}\,, \X u_{2}) = 0\,, \left(u_{1}\,, u_{2} \in \mathbb{D}^{n}\right)\right\}\,.
\end{equation*}
As explained above, we have
\begin{equation*}
\mathfrak{u} = \left\{\H \in \Mat(n)\,, \H = -\H^{*}\right\}\,, \qquad \mathfrak{p} = \left\{\H \in \Mat(n)\,, \H = \H^{*}\right\}\,,
\end{equation*}
and 
\begin{equation*}
\left[\mathfrak{u}\,, \mathfrak{u}\right] \subseteq \mathfrak{u}\,, \qquad \left[\mathfrak{u}\,, \mathfrak{p}\right] \subseteq \mathfrak{p}\,, \qquad \left[\mathfrak{p}\,, \mathfrak{p}\right] \subseteq \mathfrak{u}\,.
\end{equation*}
\end{enumerate}
\noindent The set $\mathfrak{p}$ is known as the set of Hermitian matrices. Similarly, we denote by $\U$ the real Lie group given by
\begin{equation*}
\U = \left\{g \in \GL(n)\,, \B_{\J}(g u_{1}\,, g u_{2}) = \B_{\J}(u_{1}\,, u_{2})\,, \left(u_{1}\,, u_{2} \in \mathbb{D}^{n}\right)\right\}\,.
\end{equation*}
i.e.
\begin{equation*}
\U = \left\{g \in \GL(n)\,, g^{*}g = \Id_{n}\right\}\,.    
\end{equation*}
Using the notations of \cite{KNAPP}, we have
\begin{equation*}
\U_{\J} = \begin{cases} \O(p\,, q) & \text{ if } \mathbb{D} = \mathbb{R} \\ \U(p\,, q) & \text{ if } \mathbb{D} = \mathbb{C} \\ \Sp(p\,, q) & \text{ if } \mathbb{D} = \mathbb{H} \end{cases}\,, \qquad \U = \begin{cases} \O(n) & \text{ if } \mathbb{D} = \mathbb{R} \\ \U(n) & \text{ if } \mathbb{D} = \mathbb{C} \\ \Sp(n) & \text{ if } \mathbb{D} = \mathbb{H}\end{cases}\,.
\end{equation*}
Therefore, $\U_{\J}$ is connected if and only if $\mathbb{D} \in \left\{\mathbb{C}\,, \mathbb{H}\right\}$, and compact if and only if $\J = \Id_{n}$\,.
\end{remark}

\begin{definition}

Let $\A \in \Mat(n)$\,. We say that $\A$ is positive if $\B_{\J}(\A u\,, u) > 0$ for all non-zero vector $u \in \mathbb{D}^{n}$\,.

\end{definition}

\noindent It is easy to see that if $\A$ is positive, then $0 \notin \Spec(\A)$. We denote by $\mathscr{P}$ the set of positive hermitian matrices, i.e.
\begin{equation*}
\mathscr{P} := \left\{\X \in \mathfrak{p}\,, \X \text{ is positive}\right\}\,.
\end{equation*}

\begin{notation}

If a matrix $\X \in \Mat(n)$ is positive, we write $\X > 0$\,.

\end{notation}

\noindent We denote by $\exp$ the exponential map on $\Mat(n)$, i.e. the map given by
\begin{equation*}
\Mat(n) \ni \A \to \exp(\A) := \sum\limits_{k = 0}^{\infty} \frac{\A^{k}}{k!} \in \GL(n)\,.
\end{equation*}
It is known (see \cite[Chapter~1]{BHATIA}) that the restriction of $\exp$ to $\mathfrak{p}$ is injective and such that $\exp(\mathfrak{p}) = \mathscr{P}$. In other words, 
\begin{equation*}
\exp: \mathfrak{p} \to \mathscr{P}
\end{equation*}
is bijective.

\begin{remark}

The goal is now to define a cone in $\mathfrak{p}_{\J}$ that is analogue to the cone $\mathscr{P}$ in $\mathfrak{p}$. It is natural to try to define $\mathscr{P}_{\J} \subset \mathfrak{p}_{\J}$ in direct analogy with the classical cone $\mathscr{P} \subset \mathfrak{p}$, namely as the image of $\mathfrak{p}_{\J}$ under the exponential map. However, this approach fails: the restriction of the exponential map to $\mathfrak{p}_{\J}$ is not injective, and hence cannot be bijective onto its image. Indeed, let $\X$ be the matrix in $\Mat(2)$ given by
\begin{equation*}
\X = \begin{bmatrix} 0 & 2i\pi \\ 2i\pi & 0 \end{bmatrix}\,.
\end{equation*}
For $\J = \Id_{1\,, 1}$, we have
\begin{equation*}
\J\X^{*}\J = \begin{bmatrix} 1 & 0 \\ 0 & -1 \end{bmatrix}\begin{bmatrix} 0 & -2i\pi \\ -2i\pi & 0 \end{bmatrix}\begin{bmatrix} 1 & 0 \\ 0 & -1 \end{bmatrix} = \begin{bmatrix} 0 & 2i\pi \\ 2i\pi & 0 \end{bmatrix} = \X
\end{equation*}
i.e. $\X \in \mathfrak{p}_{\J}$. Using that $\exp(\X) = \Id_{2} = \exp(0)$, we get that the exponential map is not injective on $\mathfrak{p}_{\J}$. Consequently, the exponential map cannot be used to parametrize a positive cone inside $\mathfrak{p}_{\J}$, and a different construction adapted to the $\J$-structure is required. 

\end{remark}

\noindent The following lemma is key in our paper: it shows that $\mathfrak{p}_{\J}$ can be canonically identified with the space of Hermitian matrices\,.

\begin{lemma}

The map
\begin{equation*}
\Phi_{\J}: \mathfrak{p}_{\J} \ni \X \to \J\X \in \mathfrak{p}
\end{equation*}
is well-defined and bijective\,.

\end{lemma}

\begin{proof}

We first check that $\Phi_{\J}(\X)=\J\X$ maps $\mathfrak{p}_{\J}$ into $\mathfrak{p}$. Let $\H \in \mathfrak{p}_{\J}$, so $\H^{*}=\J\H\J$. Then
\begin{equation*}
\left(\J\H\right)^{*} = \H^{*}\J^{*} = \H^{*}\J = \left(\J\H\J\right)\J = \J\H\J^{2} = \J\H\,,
\end{equation*}
hence $\J\H\in\mathfrak{p}$.
 The map $\Phi_{\J}$ is injective since $\J$ is invertible. To prove surjectivity, let $\P\in\mathfrak{p}$ and set $\X : =\J\P$. Then
\begin{equation*}
\X^{*} = \left(\J\P\right)^{*} = \P^{*}\J^{*} = \P\J\,,
\qquad \J\X\J = \J\left(\J\P\right)\J = \P\J\,.
\end{equation*}
Therefore $\X^{*} = \J\X\J$, i.e. $\X \in \mathfrak{p}_{\J}$, and
\begin{equation*}
\Phi_{\J}(\X) = \J\X = \J\left(\J\P\right) = \P\,.
\end{equation*}
This proves that $\Phi_{\J}$ is surjective, hence bijective\,.

\end{proof}

\begin{remark}

The map $\Phi_{\J}: \Mat(n) \to \Mat(n)$ given by $\Phi_{\J}(\X) = \J\X$ is such that $\Phi_{\J}(\mathfrak{u}_{\J}) = \mathfrak{u}$ and $\Phi_{\J}(\mathfrak{p}_{\J}) = \mathfrak{p}$. However, the map $\Phi_{\J}$ is not a homomorphism of Lie algebras\,.

\end{remark}

\noindent In particular, we get a bijective map 
\begin{equation}
\exp \circ \Phi_{\J}: \mathfrak{p}_{\J} \to \mathscr{P}\,.
\label{FirstBijection}
\end{equation}
Let $\Gamma_{\J}: \mathscr{P} \to \Mat(n)$ be the map given by
\begin{equation*}
\Gamma_{\J}(\X) = \J\X\,, \qquad \qquad \left(\X \in \mathscr{P}\right)\,,
\end{equation*}
and let $\mathscr{P}_{\J} := \Im(\Gamma_{\J})$. By construction, the map $\Gamma_{\J}: \mathscr{P} \to \mathscr{P}_{\J}$ is bijective, and we denote by $\exp_{\J}$ the map given by
\begin{equation*}
\begin{tikzcd}[row sep=3.2em, column sep=4.2em]
\mathfrak{p}_{\J} \arrow[r, "\Phi_{\J}"] \arrow[d, "\exp_{\J}"'] 
  & \mathfrak{p} \arrow[d, "\exp"] \\
\mathscr{P}_{\J}  & \mathscr{P} \arrow[l, "\Gamma_{\J}"']
\end{tikzcd}
\end{equation*}
In particular, we have
\begin{equation}
\exp_{\J}(\X) = \J\exp(\J\X)\,, \qquad \left(\X \in \mathfrak{p}_{\J}\right)\,.
\label{JExponentialMap}
\end{equation}

\noindent We now give a description of the set $\mathscr{P}_{\J} := \exp_{\J}(\mathfrak{p}_{\J})$. We first introduce the definition of $\J$-positive hermitian matrices\,.

\begin{definition}

We say that a matrix $\H \in \mathfrak{p}_{\J}$ is $\J$-positive if $\B(\H x\,, x) > 0$ for all non-zero $x$ in $\mathbb{D}^{n}$\,.

\end{definition}

\begin{remark}

\begin{enumerate}
\item In other words, we say that $\H \in \mathfrak{p}_{\J}$ is $\J$-positive if $x^{*}\H^{*}\J x > 0$ for all non-zero $x \in \mathbb{D}^{n}$\,.
\item A matrix $\H$ is in $\mathscr{P}_{\J}$ if and only if $\J\H \in \mathscr{P}$\,.
\end{enumerate}

\label{RemarkPositiveMatrix}

\end{remark}

\noindent The following proposition gives a nice parametrization of $\J$-positive matrices\,.

\begin{proposition}

Let $\H \in \Mat(n)$\,.
\begin{enumerate}
\item The matrix $\H$ is $\J$-Hermitian (i.e. $\H^{\sharp} = \H$, equivalently $\H^{*} = \J\H\J$) if and only if, with respect to the decomposition $\mathbb{D}^{n} = \mathbb{D}^{p} \oplus \mathbb{D}^{q}$, it can be written as
\begin{equation}
\label{EquationBlockHermitian}
\H = \begin{bmatrix} \A & \B \\ -\B^{*} & \D \end{bmatrix}\,,
\end{equation}
where $\A \in \Mat(p)$ and $\D \in \Mat(q)$ are Hermitian (i.e. $\A^{*} = \A$ and $\D^{*} = \D$), and $\B \in \Mat(p \times q)$ is arbitrary\,.

\item If, in addition, $\H$ is $\J$-positive (i.e. $\J\H>0$), then $\A$ is positive, $\D$ is negative (i.e. $(-\D) > 0$), and the following Schur complement conditions hold:
\begin{equation*}
(-\D) - \B^{*}\A^{-1}\B > 0\,.
\end{equation*}
\end{enumerate}

\end{proposition}

\begin{proof}
\begin{enumerate}
\item Let $\H \in \mathfrak{p}_{\J}$. We write $\H$ in block form relative to the decomposition $\mathbb{D}^{n} = \mathbb{D}^{p} \oplus \mathbb{D}^{q}$, i.e.
\begin{equation*}
\H = \begin{bmatrix} \H_{1, 1} & \H_{1, 2} \\ \H_{2, 1} & \H_{2, 2}\end{bmatrix}\,,
\end{equation*}
with $\H_{1, 1} \in \Mat(p)\,, \H_{1, 2} \in \Mat(p \times q)\,, \H_{2, 1} \in \Mat(q \times p)\,,$ and $\H_{2, 2} \in \Mat(q)$.
Then
\begin{equation*}
\H^{*} = \begin{bmatrix} \H^{*}_{1, 1} & \H^{*}_{2, 1} \\ \H^{*}_{1, 2} & \H^{*}_{2, 2}\end{bmatrix}\,,
\end{equation*}
and 
\begin{equation*}
\J\H\J = \begin{bmatrix} \Id_{p} & 0 \\ 0 & -\Id_{q}\end{bmatrix} \begin{bmatrix} \H_{1, 1} & \H_{1, 2} \\ \H_{2, 1} & \H_{2, 2}\end{bmatrix} \begin{bmatrix} \Id_{p} & 0 \\ 0 & -\Id_{q}\end{bmatrix}
= \begin{bmatrix} \H_{1, 1} & -\H_{1, 2} \\ -\H_{2, 1} & \H_{2,2}\end{bmatrix}\,.
\end{equation*}
Hence the condition $\H^{*} = \J\H\J$ is equivalent to the system of equalities
\begin{equation*}
\H^{*}_{1, 1} = \H_{1, 1}\,, \qquad \H^{*}_{2, 2} = \H_{2, 2}\,, \qquad \H^{*}_{2, 1} = -\H_{1, 2}\,, \qquad \H^{*}_{1, 2} = -\H_{2, 1}\,.
\end{equation*}
Therefore, $\H_{1, 1} = \H^{*}_{1, 1}\,, \H_{2, 2} = \H^{*}_{2, 2}\,,$ and $\H_{2, 1} = -\H^{*}_{1, 2}$. Conversely, any matrix of the form \eqref{EquationBlockHermitian} clearly satisfies $\H^{*} = \J\H\J$, proving (1)\,.
\item Assume now that $\H$ is $\J$-positive, i.e. $\J\H > 0$. Using the block form obtained in (1), we get 
\begin{equation*}
\J\H = \begin{bmatrix} \Id_{p} & 0 \\ 0 & -\Id_{q} \end{bmatrix} \begin{bmatrix} \A & \B \\ -\B^{*} & \D \end{bmatrix} = \begin{bmatrix} \A & \B \\ \B^{*} & -\D\end{bmatrix}\,.
\end{equation*}
This matrix is Hermitian. Since $\J\H>0$, its leading principal block must be positive, hence $\A > 0$. Likewise, the trailing principal block must be positive definite, hence $-\D > 0$, i.e.\ $\D < 0$. Finally, since $\A > 0$ is invertible, the Schur complement criterion for positive definiteness (see \cite{HORN}) applied to the Hermitian block matrix $\J\H$ gives
\begin{equation*}
\J\H > 0 \qquad \Longleftrightarrow \qquad \A > 0 \quad \text{ and } \quad \left(-\D\right) -\B^{*}\A^{-1}\B > 0\,.
\end{equation*}
This proves (2)\,.
\end{enumerate}

\end{proof}

\begin{proposition}

The set $\mathscr{P}_{\J}$ is the set of $\J$-positive matrices in $\mathfrak{p}_{\J}$\,.

\end{proposition}

\begin{proof}

Let $\H \in \mathfrak{p}_{\J}$, i.e. $\H^{*}\J = \J\H$. Using that $\exp(\A)^{*} = \exp(\A^{*})$ for all $\A \in \Mat(n)$, we get
\begin{eqnarray*}
\exp_{\J}(\H)^{*}\J & = & \left(\J\exp(\J\H)\right)^{*}\J = \exp(\J\H)^{*}\J^{*}\J = \exp\left(\left(\J\H\right)^{*}\right)\J^{2} \\
& = & \exp(\H^{*}\J^{*}) = \exp(\H^{*}\J) = \exp(\J\H)
\end{eqnarray*}
and
\begin{equation*}
\J\exp_{\J}(\H) = \J\left(\J\exp(\J\H)\right) = \J^{2}\exp(\J\H) = \exp(\J\H)\,,
\end{equation*}
i.e. $\exp_{\J}(\H) \in \mathfrak{p}_{\J}$. Moreover, for all non-zero $x \in \mathbb{D}^{n}$, we have
\begin{equation*}
x^{*}\left(\J\exp_{\J}(\H)\right)x = x^{*}\exp(\J\H)x > 0
\end{equation*}
because $\J\H \in \mathfrak{p}$, i.e. $\exp(\J\H) \in \mathscr{P}$. Therefore, $\exp_{\J}(\H)$ is $\J$-positive, i.e.
\begin{equation*}
\mathscr{P}_{\J} \subseteq \left\{\X \in \mathfrak{p}_{\J}\,, \X \text{ is J-positive}\right\}\,.
\end{equation*}
Now let $\H$ be a $\J$-positive matrix in $\mathfrak{p}_{\J}$. Then $\J\H \in \mathscr{P}$. Therefore, it follows from Equation \eqref{FirstBijection} that there exists $\X \in \mathfrak{p}_{\J}$ such that $\J\H = \exp(\J\X)$. By multiplying the previous equation by $\J$ on the left, it follows that $\H = \J\exp(\J\X) = \exp_{\J}(\X)$. Hence $\H \in \exp_{\J}(\mathfrak{p}_{\J})$, and the proposition follows\,.

\end{proof}

\begin{remark}

Let $\X \in \mathfrak{p}$, i.e. $\X = \X^{*}$. As explained above, we have $\J\X \in \mathfrak{p}_{\J}$. One can easily see that $\X\J$ is also in $\mathfrak{p}_{\J}$. Indeed,
\begin{equation*}
\left(\X\J\right)^{\sharp} = \J\X^{\sharp} = \J\left(\J\X^{*}\J\right) = \X^{*}\J = \X\J\,,
\end{equation*}
i.e. $\X\J \in \mathfrak{p}_{\J}$. Moreover, if $\X > 0$, then $\X\J$ is $\J$-positive. Indeed, for all $x \in \mathbb{D}^{n}$, we have
\begin{equation*}
x^{*}\J(\X\J)x = \left(\J x\right)^{*}\X\left(\J x\right) > 0\,,
\end{equation*}
i.e. $\X\J \in \mathscr{P}_{\J}$. Therefore, as a set, we have
\begin{equation*}
\mathscr{P}_{\J} = \J\mathscr{P} = \mathscr{P}\J\,.
\end{equation*}

\label{RemarkFebruary22}

\end{remark}

\noindent The cone $\mathscr{P}_{\J}$ is not a subgroup of $\GL(n)$. Indeed, it is not stable under multiplication. However, the following lemma shows that the inverse of an element in $\mathscr{P}_{\J}$ still lies in $\mathscr{P}_{\J}$\,.  

\begin{lemma}

\begin{enumerate}
\item If $\H \in \mathscr{P}_{\J}$, then $\H^{-1} \in \mathscr{P}_{\J}$\,.
\item If $\H = \exp_{\J}(\X)$ with $\X \in \mathfrak{p}_{\J}$, then in general $\H^{-1} \neq \exp_{\J}(-\X)$. More precisely,
\begin{equation}
\label{EquationInverseExponential}
\exp_{\J}(\X)^{-1} = \exp(-\J\X)\J\,,
\end{equation}
and $\exp_{\J}(\X)^{-1} = \exp_{\J}(-\X)$ holds if and only if $\exp(-\J\X)$ commutes with $\J$\,.
\end{enumerate}

\label{InversePJ}
\end{lemma}

\begin{proof}
\begin{enumerate}
\item Let $\H \in \mathscr{P}_{\J}$. By definition, $\H \in \mathfrak{p}_{\J}$ and $\J\H > 0$. Since $\H^{\sharp} = \H$, taking inverses yields
\begin{equation*}
\left(\H^{-1}\right)^{\sharp} = \left(\H^{\sharp}\right)^{-1} = \H^{-1}\,,
\end{equation*}
so $\H^{-1} \in \mathfrak p_{\J}$. Using the map $\Phi_{\J}$, we can write $\H = \J\P$, with $\P \in \mathscr{P}$. In particular, 
\begin{equation*}
\H^{-1} = \P^{-1}\J\,,
\end{equation*}
i.e. $\J\H^{-1} = \J\left(\P^{-1}\J\right)=\J\P^{-1}\J$. Since $\P^{-1} > 0$, hence for all non-zero $x \in \mathbb{D}^{n}$, we get
\begin{equation*}
x^{*}\J\H^{-1}x = x^{*}\J\P^{-1}\J x = \left(\J x\right)^{*}\P^{-1}(\J x) >0\,, 
\end{equation*}
i.e. $\H^{-1} \in \mathscr{P}_{\J}$\,.
\item Let $\H = \exp_{\J}(\X)=\J\exp(\J\X)$, with $\X \in \mathfrak{p}_{\J}$. Then
\begin{equation*}
\H^{-1} = \left(\J\exp(\J\X)\right)^{-1} = \exp(\J\X)^{-1}\J^{-1}
=\exp(-\J\X)\J\,,
\end{equation*}
which proves \eqref{EquationInverseExponential}. On the other hand,
\begin{equation*}
\exp_{\J}(-\X)=\J\exp(\J(-\X))=\J\exp(-\J\X)\,.
\end{equation*}
Hence $(\exp_{\J}(\X))^{-1} = \exp_{\J}(-\X)$ if and only if
\begin{equation*}
\exp(-\J\X)\J = \J\exp(-\J\X)\,,
\end{equation*}
i.e. if and only if $\exp(-\J\X)$ commutes with $\J$\,. 
\end{enumerate}

\end{proof}

\begin{example}

We give an explicit example showing that $\exp_{\J}(\X)^{-1} \neq \exp_{\J}(-\X)$. Let $p = q = 1$, and let $\X$ be the matrix in $\Mat(2)$ given by
\begin{equation*}
\X = \begin{bmatrix} 0 & i \\ i & 0\end{bmatrix}\,.
\end{equation*}
Using that
\begin{equation*}
\J\X = \begin{bmatrix} 1 & 0 \\ 0 & -1 \end{bmatrix}\begin{bmatrix} 0 & i \\ i & 0\end{bmatrix} = \begin{bmatrix} 0 & i \\ -i & 0\end{bmatrix}\,,
\end{equation*}
we get that $\J\X \in \mathfrak{p}$, so $\X \in \mathfrak{p}_{\J}$. Let $\K := \J\X$. One can see that $\K^{2} = \Id_{2}$, hence
\begin{equation*}
\exp(-\K) = \begin{bmatrix} \cosh(1) & -i\sinh(1)\\ i\sinh(1) & \cosh(1) \end{bmatrix}\,.
\end{equation*}
Using Equation \eqref{EquationInverseExponential}, we obtain
\begin{equation*}
\exp_{\J}(\X)^{-1} = \exp(-\K)\J = \begin{bmatrix} \cosh(1) & i\sinh(1)\\ i\sinh(1) & -\cosh(1) \end{bmatrix}\,,
\end{equation*}
whereas
\begin{equation*}
\exp_{\J}(-\X) = \J\exp(-\K) = \begin{bmatrix} \cosh(1) & -i\sinh(1)\\
-i\sinh(1) & -\cosh(1) \end{bmatrix}\,.
\end{equation*}
These matrices are different, which shows that $\exp_{\J}(\X)^{-1}\neq \exp_{\J}(-\X)$ in general\,.

\end{example}

\begin{remark}

If $\H \in \mathfrak{p}$, then $\Spec(\H) \subseteq \mathbb{R}$. However, for a matrix $\X \in \mathfrak{p}_{\J}$, we do not get, in general, that $\Spec(\X) \subseteq \mathbb{R}$. Indeed, let $\X = \begin{bmatrix} 0 & 1 \\ -1 & 0 \end{bmatrix} \in \Mat(2)$. We have
\begin{equation*}
\X^{*}\J = \begin{bmatrix} 0 & 1 \\ 1 & 0 \end{bmatrix} = \J\X\,,
\end{equation*}
i.e. $\X \in \mathfrak{p}_{\J}$. However, $\Spec(\X) = \left\{\pm i\right\}$\,.

\end{remark}

\begin{remark}

\noindent We finish this section with a remark. We denote by $\bullet$ the binary operation on $\Mat(n)$ given by
\begin{equation}
\A \bullet \B := \A\J\B\,, \qquad \left(\A\,,\B \in \Mat(n)\right)\,.
\label{JProduct}
\end{equation}
This operation $\bullet$ is associative. Indeed, for all $\A\,, \B\,, \C \in \Mat(n)$, we have
\begin{equation*}
\left(\A \bullet \B\right) \bullet \C = \left(\A\J\B\right)\J\C = \A\J\left(\B\J\C\right) = \A \bullet \left(\B \bullet \C\right)\,.
\end{equation*}
The neutral element for $\bullet$ is precisely $\J$. Indeed, for all $\A \in \Mat(n)$,
\begin{equation*}
\J \bullet \A = \J\left(\J\A\right) = \J^{2}\A = \A \qquad \text{and} \qquad \A \bullet \J = \left(\A\J\right)\J = \A\,.
\end{equation*}
An element $\A$ is invertible for the $\J$-product if and only if it is invertible as a matrix, and we have
\begin{equation*}
\A \bullet \left(\J\A^{-1}\J\right) = \A\J\left(\J\A^{-1}\J\right) = \A\J^{2}\A^{-1}\J = \A\A^{-1}\J = \J\,,
\end{equation*}
i.e. the inverse of $\A$ with respect to $\bullet$ is $\J\A^{-1}\J$. 

\noindent Moreover, for all $\A\,, \B \in \Mat(n)$, we have
\begin{equation*}
\left(\A \bullet \B\right)^{\sharp} = \J\left(\A\J\B\right)^{*}\J = \J\B^{*}\J\A^{*}\J = \left(\J\B^{*}\J\right)\J\left(\J\A^{*}\J\right) = \B^{\sharp} \bullet \A^{\sharp}\,.
\end{equation*}

\noindent Let $\left[\cdot\,, \cdot\right]$ be the binary operation on $\Mat(n)$ given by
\begin{equation*}
\left[\X\,, \Y\right]_{\J} = \X \bullet \Y - \Y \bullet \X\,, \qquad \left(\X\,, \Y \in \Mat(n)\right)\,.
\end{equation*}
One can see that the map
\begin{equation*}
\widetilde{\Phi_{\J}}: \left(\Mat(n)\,, \left[\cdot\,, \cdot\right]\right) \ni \X \to \J\X \in \left(\Mat(n)\,, \left[\cdot\,, \cdot\right]_{\J}\right)
\end{equation*}
is an isomorphism of Lie algebras\,.

\noindent The map
\begin{equation*}
\exp_{\J}: \mathfrak{p}_{\J} \to \mathscr{P}_{\J}
\end{equation*}
is exactly the exponential map associated with the $\J$-product. Indeed, for all $\X\,, \Y \in \mathfrak{p}_{\J}$, we get 
\begin{equation*}
\exp_{\J}(\X) \bullet \exp_{\J}(\Y) = \exp_{\J}\left(\J \cdot \BCH_{\J}(\J\X\,, \J\Y)\right)\,,
\end{equation*}
where $\BCH_{\J}$ is the standard Baker-Campbell-Hausdorff formula (see \cite{KNAPP}) where the Lie bracket is replaced by $\left[\cdot\,, \cdot\right]_{\J}$.

\noindent Finally, one can see that for all $g \in \GL(n)$, there exists a unique pair $\left(k\,, p\right) \in \U \times \mathscr{P}_{\J}$ such that $g = k \bullet p$. Indeed, using the classical polar decomposition of $\GL(n)$ (see \cite{BHATIA}), there exists a unique pair $\left(k\,, \widetilde{p}\right) \in \U \times \mathscr{P}$ such that $g = k\widetilde{p}$. Therefore, using that $\mathscr{P}_{\J} = \J\mathscr{P}$ and $\J^{2} = \Id_{n}$, we get
\begin{equation*}
g = k\widetilde{p} = k \J^{2}\widetilde{p} = k\J\left(\J\widetilde{p}\right) = k \bullet p\,, 
\end{equation*}
with $p = \J\widetilde{p} \in \mathscr{P}_{\J}$\,.

\end{remark}

\section{Square root on $\mathscr{P}_{\J}$ and properties}

\label{SectionTwo}

Let $\exp_{\J}: \mathfrak{p}_{\J} \to \mathscr{P}_{\J}$ be the $\J$-exponential map defined in Equation \eqref{JExponentialMap}. As explained in Section \ref{SectionOne}, the map $\exp_{\J}$ is bijective. Let $\log_{\J} := \exp^{-1}_{\J}: \mathscr{P}_{\J} \to \mathfrak{p}_{\J}$ be the inverse of $\exp_{\J}$. In particular, we have
\begin{equation*}
\log_{\J}(\X) = \J\log(\J\X)\,, \qquad \qquad \left(\X \in \mathscr{P}_{\J}\right)\,,
\end{equation*}
where $\log: \mathscr{P} \to \mathfrak{p}$ is the inverse of $\exp: \mathfrak{p} \to \mathscr{P}$\,.

\begin{definition}

For all $t \in \mathbb{R}$ and $\X \in \mathscr{P}_{\J}$, we denote by $\X^{t}_{\J}$ the matrix in $\mathscr{P}_{\J}$ given by
\begin{equation*}
\X^{t}_{\J} = \exp_{\J}\left(t\log_{\J}(\X)\right)\,.
\end{equation*}
If $t = \frac{1}{2}$, the matrix $\X^{\frac{1}{2}}_{\J}$ is the $\J$-square root of $\X$\,.

\end{definition}

\begin{remark}

For all $\A \in \mathscr{P}$ and $t \in \mathbb{R}$, we denote by $\A^{t}$ the matrix in $\mathscr{P}$
given by
\begin{equation*}
\A^{t} = \exp(t\log(\A))\,.
\end{equation*}
We have the following properties (see \cite{BHATIA}): for all $\A\,, \B \in \mathscr{P}$ and $s\,, t \in \mathbb{R}$:
\begin{enumerate}
\item $\A^{s}\A^{t} = \A^{s+t}$ and $\left(\A^{s}\right)^{t} = \A^{st}$\,.
\item If $\A$ and $\B$ commute in $\mathscr{P}$, then $\A\B \in \mathscr{P}$ and 
\begin{equation*}
\left(\A\B\right)^{s} = \A^{s}\B^{s}\,.
\end{equation*}
\item If $\C \in \U$, then
\begin{equation*}
\left(\C\A\C^{*}\right)^{t} = \C\A^{t}\C^{*}\,.
\end{equation*}
\end{enumerate}
\label{RemarkSquareRootInP}

\end{remark}

\noindent We now prove some interesting properties for our matrices $\X^{t}_{\J}$\,.

\begin{proposition}

\begin{enumerate}
\item For all $\X \in \mathscr{P}_{\J}$ and $t \in \mathbb{R}$, we have $\X^{t}_{\J} = \J(\J\X)^{t}$. In particular, $\X^{\frac{1}{2}}_{\J} = \J(\J\X)^{\frac{1}{2}}$ and $(\X^{\frac{1}{2}}_{\J})^{2}_{\J} = \X$\,.
\item For all $s\,, t\in\mathbb{R}$ and $\X \in \mathscr{P}_{\J}$, we have
\begin{equation*}
\X^{0}_{\J} = \J\,, \qquad \X^{1}_{\J} = \X\,, \qquad (\X^{t}_{\J})^{s}_{\J} = \X^{ts}_{\J}\,, \qquad \log_{\J}(\X^{t}_{\J}) = t\log_{\J}(\X)\,.
\end{equation*}
\item For all $\X \in \mathscr{P}_{\J}$, we have
\begin{equation*}
(\X^{-1})^{t}_{\J} = (\X^{t}_{\J})^{-1}\,,
\end{equation*}
In particular, $(\X^{-1})^{\frac{1}{2}}_{\J} = (\X^{\frac{1}{2}}_{\J})^{-1}$\,.
\item For all $\X \in \mathscr{P}_{\J}$ and $s,\alpha \in \mathbb{R}$, we get that 
\begin{equation*}
\X^{s}_{\J}=\X^{\alpha s}_{\J} \bullet \X^{(1-\alpha) s}_{\J} \qquad \text{ and } \qquad\X^{s}_{\J} \bullet \X^{-s}_{\J}=\J\,.
\end{equation*}
\end{enumerate}

\label{PropositionSquareRoot}

\end{proposition}

\begin{proof}

\begin{enumerate}
\item Using that $\exp_{\J}(\H) = \J\exp(\J\H)$ and $\log_{\J}(\X) = \J\log(\J\X)$, we get
\begin{equation*}
\X^{t}_{\J} = \exp_{\J}\left(t\log_{\J}(\X)\right) = \J\exp\left(t\log(\J\X)\right) = \J(\J\X)^{t}\,.
\end{equation*}
Moreover, we get
\begin{equation*}
(\X^{\frac{1}{2}}_{\J})^{2}_{\J} = \J\left(\J\X^{\frac{1}{2}}_{\J}\right)^{2} = \J\left(\J^{2}(\J\X)^{\frac{1}{2}}\right)^{2} = \J\left(\J\X\right)^{\frac{1}{2}}\left(\J\X\right)^{\frac{1}{2}} = \J(\J\X) = \J^{2}\X = \X\,.
\end{equation*}
\item From (1), we get that for all $\X \in \mathscr{P}_{\J}$, $\X^{0}_{\J} = \J(\J\X)^{0} = \J$ and $\X^{1}_{\J} = \J(\J\X)^{1} = \J^{2}\X = \X$. Moreover, for all $s\,, t \in \mathbb{R}$,
\begin{equation*}
\left(\X^{t}_{\J}\right)^{s}_{\J} = \J\left(\J\X^{t}_{\J}\right)^{s} = \J\left(\J^{2}\left(\J\X\right)^{t}\right)^{s} = \J\left(\left(\J\X\right)^{t}\right)^{s} = \J(\J\X)^{ts} = \X^{ts}_{\J}\,,
\end{equation*}
and 
\begin{eqnarray*}
\log_{\J}(\X^{t}_{\J}) = \J\log(\J\X^{t}_{\J}) & = & \J\log(\J^{2}(\J\X)^{t}) = \J\log\left(\exp(t\log(\J\X))\right) \\
& = & \J\left(t\log(\J\X)\right) = t\log_{\J}(\X)\,.
\end{eqnarray*}
\item Let $\X \in \mathscr{P}_{\J}$ and let $\P := \J\X \in \mathscr{P}$. Then $\X = \J\P$ and
\begin{equation*}
\X^{-1} = \left(\J\P\right)^{-1} = \P^{-1}\J\,.
\end{equation*}
Hence $\J\X^{-1} = \J(\P^{-1}\J) = \J\P^{-1}\J$. Since $\J$ is unitary (i.e. $\J^{*} = \J = \J^{-1}$), it follows from Remark \ref{RemarkSquareRootInP} that
\begin{equation*}
\left(\J\P^{-1}\J\right)^t = \J\left(\P^{-1}\right)^{t}\J = \J\P^{-t}\J\,.
\end{equation*}
It follows that
\begin{equation*}
\left(\X^{-1}\right)^{t}_{\J} = \J\left(\J\X^{-1}\right)^{t} =
\J\left(\J\P^{-t}\J\right) = \P^{-t}\J\,.
\end{equation*}
On the other hand, from (1) we have $\X^{t}_{\J} = \J\left(\J\X\right)^{t} = \J\P^{t}$, so
\begin{equation*}
\left(\X^{t}_{\J}\right)^{-1} = \left(\J\P^{t}\right)^{-1} = \left(\P^{t}\right)^{-1}\J = \P^{-t}\J\,,
\end{equation*}
therefore
\begin{equation*}
\left(\X^{-1}\right)^{t}_{\J} = \left(\X^{t}_{\J}\right)^{-1}\,.
\end{equation*}
\item Using that $\X^{s}_{\J} = \J(\J\X)^{s}$, we get
\begin{eqnarray*}
\X^{\alpha s}_{\J} \bullet \X^{(1-\alpha)s}_{\J} & = & \J (\J\X)^{\alpha s}\J^{2}(\J\X)^{(1-\alpha)s} = \J\left(\J\X\right)^{\alpha s}\left(\J\X\right)^{(1-\alpha)s} \\
& = & \J\left(\J\X\right)^{\alpha s + (1-\alpha)s} = \J(\J\X)^{s} = \X^{s}_{\J}\,,
\end{eqnarray*}
and
\begin{equation*}
\X^{s}_{\J}\bullet\X^{-s}_{\J} = \J(\J\X)^{s}\J^{2}(\J\X)^{-s} = \J\left(\J\X\right)^{s-s} = \J\left(\J\X\right)^{0} = \J\,.
\end{equation*}
\end{enumerate}

\end{proof}

\begin{proposition}

We have
\begin{enumerate}
\item For all $g \in \GL(n)$ and $\X \in \mathscr{P}_{\J}$, $g\X g^{\sharp} \in \mathscr{P}_{\J}$\,. 
\item The corresponding action of $\GL(n)$ on $\mathscr{P}_{\J}$ is transitive\,.
\item If $g \in \K_{\J} := \U_{\J} \cap \U$ (i.e. $g^{\sharp}g = \Id_{n}$ and $g^{*}g = \Id_{n}$), then for all $t \in \mathbb{R}$, we get
\begin{equation*}
\left(g\X g^\sharp\right)^{t}_{\J} = g\X^{t}_{\J}g^{\sharp}\,.
\end{equation*}
In particular, we get $(g\X g^{\sharp})^{\frac{1}{2}}_{\J} = g\X^{\frac{1}{2}}_{\J}g^{\sharp}$\,.
\item Let $\X\,, \Y \in \mathscr{P}_{\J}$ such that $\X \bullet \Y = \Y \bullet \X$. Then for all $t \in \mathbb{R}$, we get
\begin{equation*}
\left(\X \bullet \Y\right)^{t}_{\J} = \X^{t}_{\J} \bullet \Y^{t}_{\J}\,.
\end{equation*}
\end{enumerate}

\label{PropositionPropertiesSRJ}

\end{proposition}

\begin{proof}

\begin{enumerate}
\item Let $g \in \GL(n)$ and $\X \in \mathscr{P}_{\J}$. Using that $\X^{\sharp} = \X$, we get that
\begin{equation*}
\left(g \X g^{\sharp}\right)^{\sharp} = \left(g^{\sharp}\right)^{\sharp} \X^{\sharp} g^{\sharp} = g \X g^{\sharp}\,,
\end{equation*}
i.e. $g\X g^{\sharp} \in \mathfrak{p}_{\J}$. Moreover, using that $\X \in \mathscr{P}_{\J}$, it follows that for all $x \in \mathbb{D}^{n}$ that
\begin{equation*}
\B_{\J}(g\X g^{\sharp}x\,, x) = x^{*}\left(\J g\X \J g^{*} \J\right)x = \left(\J g^{*}\J x\right)^{*} \underbrace{\J\X}_{> 0} \left(\J g^{*}\J x\right)\,,
\end{equation*}
i.e. $g\X g^{\sharp}$ is $\J$-positive. Therefore $g\X g^{\sharp} \in \mathscr{P}_{\J}\,.$
\item  The action of $\GL(n)$ on $\mathscr{P}$ is transitive (see \cite{BHATIA} and Appendix \ref{AppendixA} for the case $\mathbb{D} = \mathbb{H}$). In particular, for all $\P \in \mathscr{P}$, there exists $g \in \GL(n)$ such that $\P = gg^{*}$, i.e. $\mathscr{P} = \left\{gg^{*}\,, g \in \GL(n)\right\}$. Using Remark \ref{RemarkFebruary22}, we have $\mathscr{P}_{\J} = \J\mathscr{P} = \mathscr{P}\J$. For all $g \in \GL(n)$, we get
\begin{equation*}
g\J g^{\sharp} = g\J\left(\J g^{*}\J\right) = g\J^{2}g^{*}\J = \underbrace{gg^{*}}_{\in \mathscr{P}}\J\,,
\end{equation*}
i.e. $\mathscr{P}_{\J} \subseteq \left\{g\J g^{\sharp}\,, g \in \GL(n)\right\} = \mathscr{P}_{\J}$, i.e. $\GL(n) \curvearrowright \mathscr{P}_{\J}$ is transitive\,.
\item Let $g \in \U \cap \U_{\J}$. From the definition of $\X^{t}_{\J}$, we obtain $(g\X g^\sharp)^{t}_{\J} = \J\left(\J(g\X g^{\sharp})\right)^{t}$. Moreover, using that $g \in \U_{\J}$, we have $g^{\sharp}g = \Id_{n}$, i.e. $\left(g^{-1}\right)^{*}\J = \J g$. Therefore $g^{*}\J = \J g^{-1}$ and
\begin{equation*}
\J(g\X g^{\sharp}) = \J g\X\J g^{*}\J = \left(g^{-1}\right)^{*}\J\X\J^{2}g^{-1} = \left(g^{-1}\right)^{*}\J\X g^{-1}\,.
\end{equation*}
Using that $\J\X \in \mathscr{P}$ and $g \in \U$, it follows that
\begin{equation*}
\left(\J(g\X g^{\sharp})\right)^{t} = \left(g^{-1}\right)^{*}(\J\X)^{t}g^{-1}\,.
\end{equation*}
Multiplying the previous equation by $\J$ on the left, we obtain
\begin{equation*}
(g\X g^{\sharp})^{t}_{\J} = \J \left(g^{-1}\right)^{*} (\J\X)^{t}g^{-1}\,.
\end{equation*}
Finally, it follows from $\J \left(g^{-1}\right)^{*} = g\J$ and $g^{\sharp} = g^{-1}$ that
\begin{equation*}
(g\X g^{\sharp})^{t}_{\J} = g\J(\J\X)^{t}g^{-1} = g\X^{t}_{\J}g^{-1} = g\X^{t}_{\J}g^{\sharp}\,,
\end{equation*}
which proves the desired property\,.
\item Using that $\X\J\Y = \Y\J\X$, it follows that $(\J\X)(\J\Y) = (\J\Y)(\J\X)$. In particular, it follows from Remark \ref{RemarkSquareRootInP} that $\left((\J\X)(\J\Y)\right)^{t} = \left(\J\X\right)^{t}\left(\J\Y\right)^{t}$.

\noindent We get
\begin{equation*}
(\X \bullet \Y)^{t}_{\J} = \J\left(\J(\X\J\Y)\right)^{t} = \J\left((\J\X)(\J\Y)\right)^{t}
= \J\left(\J\X\right)^{t}\left(\J\Y\right)^{t}
\end{equation*}
and 
\begin{equation*}
\X^{t}_{\J} \bullet \Y^{t}_{\J} = \J\left(\J\X\right)^{t}\J^{2}\left(\J\Y\right)^{t} = \J\left(\J\X\right)^{t}\left(\J\Y\right)^{t}\,,
\end{equation*}
so the result follows\,.
\end{enumerate}

\end{proof}

\begin{remark}

The group $\K_{\J}$ defined in Proposition \ref{PropositionPropertiesSRJ} is the maximal compact subgroup of $\U_{\J}$. More precisely, we have
\begin{equation*}
\K_{\J} := \begin{cases} \O(p) \times \O(q) & \text{ if } \U_{\J} = \O(p\,, q) \\ \U(p) \times \U(q) & \text{ if } \U_{\J} = \U(p\,, q) \\ \Sp(p) \times \Sp(q) & \text{ if } \U_{\J} = \Sp(p\,, q) \end{cases}\,.
\end{equation*}

\end{remark}

\begin{remark}

The computations in this section show that $\mathscr{P}_{\J}$ is the natural analogue of the cone of positive definite Hermitian matrices, with the identity matrix replaced by $\J$ and the usual product replaced by $\bullet$\,.

\noindent It is easy to see that this yields a $*$-algebra structure on $\Mat(n)$. Notice, however, that the self-adjoint operators associated to the involution $\sharp$, i.e. elements of $\mathfrak{p}_{\J}$, do not have a real spectrum. Therefore, this negates the possibility of using functional calculus, as it would be the case in a $\C^{*}$-algebra. In the latter case, the geometric means are well-understood\,.

\end{remark}

\section{Loewner order on the cone $\mathscr{P}_{\J}$}

\label{SectionThree}

We start this section by recalling the definition of Loewner's order on $\mathscr{P}$. First of all, we denote by $\mathscr{P}_{0}$ the set of positive semi-definite Hermitian matrices in $\mathfrak{p}$, i.e.
\begin{equation*}
\mathscr{P}_{0} := \left\{\X \in \mathfrak{p}\,, \B_{\J}(\X u\,, u) \geq 0\,, \left(u \in \mathbb{D}^{n}\right)\right\}\,.
\end{equation*}
In particular, $\mathscr{P} \subseteq \mathscr{P}_{0}$, and if $\X \in \mathscr{P}_{0}$, we write $\X \geq 0$\,.

\noindent The Loewner order on $\mathscr{P}$ is the partial order defined by
\begin{equation*}
\A \preceq \B \qquad \Longleftrightarrow \qquad \B-\A \in \mathscr{P}_{0}\,,
\end{equation*}
that is $\B-\A$ is a positive semi-definite Hermitian matrix. Equivalently, $\A \preceq \B$ if and only if $v^{*}\left(\B-\A\right)v \geq 0$ for all non-zero $v\in\mathbb{D}^{n}$.

\begin{remark}

\noindent This order is closed under congruence transformations. Indeed, if $\A \preceq \B$, then for all $\M \in \Mat(n)$, we get
\begin{equation*}
\M^{*}\A\M \preceq \M^{*}\B\M\,.
\end{equation*}
Moreover, for all $t \in \left[0\,, 1\right]$, we have
\begin{equation*}
\A^{t} \preceq \B^{t}\,.
\end{equation*}

\label{LoewnerOrderP}

\end{remark}

\noindent We now defined a partial order on the cone $\mathscr{P}_{\J}$ by using $\left(\mathscr{P}\,, \preceq\right)$ and the map $\mathscr{P}_{\J} \ni \X \to \J\X \in \mathscr{P}$\,.

\begin{definition}

For all $\X\,, \Y \in \mathscr{P}_{\J}$, we say that $\X \preceq_{\J} \Y$ if and only if $\J\X \preceq \J\Y$\,.

\end{definition}

\noindent This order is precisely the pullback of the Loewner order on $\mathscr{P}$
through the linear isomorphism $\mathscr{P}_{\J} \ni \X \mapsto \J\X \in \mathscr{P}$\,.

\noindent It is easy to see that $\preceq_{\J}$ defines a partial order on $\mathscr{P}_{\J}$. The next lemmas give an analogue of the results given in Remark \ref{LoewnerOrderP} on our cone $\mathscr{P}_{\J}$\,.

\begin{lemma}

If $\X\,, \Y \in \mathscr{P}_{\J}$ are such that $\X \preceq_{\J} \Y$, then for all $t \in \left[0\,, 1\right]$, we get
\begin{equation*}
\X^{t}_{\J} \preceq_{\J} \Y^{t}_{\J}\,.
\end{equation*}

\label{LemmaMonoticity}

\end{lemma}

\begin{proof}

Let $\X \preceq_{\J} \Y$, i.e. $\J\X \preceq \J\Y$. Since $\J\X$ and $\J\Y$ are both in $\mathscr{P}$, it follows from Remark \ref{LoewnerOrderP} that for all $t \in \left[0\,, 1\right]$, we have
\begin{equation*}
(\J\X)^{t} \preceq (\J\Y)^{t}\,.
\end{equation*}
From $\J^{2} = \Id$, we get that $(\J\X)^{t} = \J^{2}(\J\X)^{t} = \J\X^{t}_{\J}$, so
\begin{equation*}
\J\X^{t}_{\J} \preceq \J\Y^{t}_{\J}\,,
\end{equation*}
which is precisely the statement that
\begin{equation*}
\X^{t}_{\J} \preceq_{\J} \Y^{t}_{\J}\,.
\end{equation*}

\end{proof}

\begin{remark}

The restriction $t \in \left[0\,, 1\right]$ in Lemma \ref{LemmaMonoticity} is necessary since the map $\A \mapsto \A^{t}$ is operator monotone only in this range\,.

\end{remark}

\begin{lemma}

If $\X\,, \Y \in \mathscr{P}_{\J}$ are such that $\X \preceq_{\J} \Y$, then for all $\C \in \Mat(n)$, we have
\begin{equation*}
\C^{\sharp}\X\C \preceq_{\J} \C^{\sharp}\Y\C\,.
\end{equation*}    

\end{lemma}

\begin{proof}

Assume $\X \preceq_{\J} \Y$. Then $\J\X \preceq \J\Y$ in $\mathscr{P}$. It follows from Remark \ref{LoewnerOrderP} that for all $\C \in \Mat(n)$, we have
\begin{equation*}
\C^{*}(\J\X)\C \preceq \C^{*}(\J\Y)\C\,.
\end{equation*}
Using that $\C^{\sharp} = \J^{-1}\C^{*}\J$, we get
\begin{equation*}
\C^{*}(\J\X)\C = \J(\C^{\sharp}\X\C)\,, \qquad \C^{*}(\J\Y)\C = \J(\C^{\sharp}\Y\C)\,.
\end{equation*}
Therefore,
\begin{equation*}
\J(\C^{\sharp}\X\C) \preceq \J(\C^{\sharp}\Y\C)\,.
\end{equation*}
By definition of $\preceq_{\J}$, this is equivalent to
\begin{equation*}
\C^{\sharp}\X\C \preceq_{\J} \C^{\sharp}\Y\C\,.
\end{equation*}

\end{proof}

\noindent The following lemma follows from the definition of the Loewner order. We record it for later use.

\begin{lemma}

If $\X\,, \Y \in \mathscr{P}_{\J}$ are such that $\X \preceq_{\J} \Y$, then 
\begin{equation*}
\X^{-1} \succeq_{\J} \Y^{-1}\,.
\end{equation*}

\label{LemmaInverse}

\end{lemma}

\section{Geodesics on the cone $\mathscr{P}_{\J}$}

\label{SectionFour}

We start this section by recalling some classical results on the cone
$\mathscr{P}$ of positive Hermitian matrices (see \cite{BHATIA} in the case $\mathbb{D} \in \left\{\mathbb{R}\,, \mathbb{C}\right\}$ and Appendix \ref{AppendixA} for $\mathbb{D} = \mathbb{H}$). For each $\P \in \mathscr{P}$, we identify the tangent space $\T_{\P}\mathscr{P}$ with $\mathfrak{p}$. The Riemannian metric on $\mathscr{P}$ is defined by
\begin{equation}
\langle \U\,, \V\rangle_{\P} = \widetilde{\tr}\left(\P^{-1}\U\P^{-1}\V\right)\,,
\qquad \left(\U\,, \V \in \T_{\P}\left(\mathscr{P}\right) \cong \mathfrak{p}\right)\,.
\label{RiemannianMetricP}
\end{equation}
with $\widetilde{\tr}: \mathbb{D} \to \mathbb{R}$ given by
\begin{equation*}
\widetilde{\tr}(\X) = \begin{cases} \tr(\X) & \text{ if } \mathbb{D} \in \left\{\mathbb{R}\,, \mathbb{C}\right\} \\ \trd(\X) & \text{ if } \mathbb{D} = \mathbb{H} \end{cases}
\end{equation*}
where $\trd$ is the reduced trace defined in Equation \ref{ReducedTrace}\,.

\noindent As explained in \cite{BHATIA} and Theorem \ref{MainTheorem},the unique geodesic joining $\P$ and $\Q$ in $\mathscr{P}$ is
\begin{equation}
\gamma(t) = \P^{\frac{1}{2}}\left(\P^{-\frac{1}{2}}\Q\P^{-\frac{1}{2}}\right)^{t} \P^{\frac{1}{2}}\,, \qquad \left(t \in \left[0\,, 1\right]\right)\,.
\label{GeodesicInP}
\end{equation}

\noindent For all $g \in \GL(n)$ and $\X \in \mathscr{P}$, we have $g\X g^{*} \in \mathscr{P}$. For all $\Q \in \GL(n)$, we denote by $\zeta_{\Q}$ the map
\begin{equation*}
\zeta_{\Q}: \mathscr{P} \ni \P \longmapsto \Q\P\Q^{*} \in \mathscr{P}\,.
\end{equation*}
For all $\P \in \mathscr{P}$, we have
\begin{equation*}
\langle (d\zeta_{\Q})_{\P}(\U)\,, (d\zeta_{\Q})_{\P}(\V)\rangle_{\zeta_{\Q}(\P)}
= \langle \U\,, \V\rangle_{\P}\,, \qquad \left(\U\,, \V \in \mathfrak{p} = \T_{\P}(\mathscr{P})\right)\,.
\end{equation*}
i.e. the Riemannian form is $\GL(n)$-invariant. 

\begin{remark}

The explicit form of the geodesics on $\mathscr{P}$ can be obtained using the Levi-Civita connection associated with the $\langle\cdot\,, \cdot\rangle$ defined in Equation \eqref{RiemannianMetricP}. 

\noindent As explained in \cite{LEE} (in the case $\mathbb{D} \in \left\{\mathbb{R}\,, \mathbb{C}\right\}$ - the proof in the quaternionic case is similar), the Levi-Civita connection on $\mathscr{P}$ is given by
\begin{equation*}
\nabla_{\U}(\V) = \D_{\U}(\V) - \frac{1}{2}\left(\U\P^{-1}\V + \V\P^{-1}\U\right)\,,
\end{equation*}
where $\D_{\U}(\V)$ denotes the directional derivative of the vector field $\V$
in the direction $\U$. A smooth curve $\gamma : \I \to \mathscr{P}$ is a geodesic if and only if it satisfies the geodesic equation
\begin{equation*}
\nabla_{\dot{\gamma}(t)}(\dot{\gamma}(t)) = 0\,,
\end{equation*}
which gives the second order differential equation
\begin{equation}
\ddot{\gamma}(t) = \dot{\gamma}(t)\gamma(t)^{-1}\dot{\gamma}(t)\,.
\label{ODEP}
\end{equation}
The existence and uniqueness of solutions of \eqref{ODEP}  follows from standard ODE theory on manifolds. The last step is to show that the curve $\gamma$ given in Equation \eqref{GeodesicInP} satisfies \eqref{ODEP}. Let $\L = \log\left(\P^{-\frac{1}{2}}\Q\P^{-\frac{1}{2}}\right) \in \mathfrak{p}$, i.e. $\gamma(t) = \P^{\frac{1}{2}}\exp(t\L)\P^{\frac{1}{2}}$. Since $\L$ is constant and $\frac{d}{dt}\exp(t\L)=\exp(t\L)\L$, we get:
\begin{equation*}
\dot{\gamma}(t) = \P^{\frac{1}{2}}\left(\frac{d}{dt}\exp(t\L)\right)\P^{\frac{1}{2}} = \P^{\frac{1}{2}}\exp(t\L)\L\P^{\frac{1}{2}}\,,
\end{equation*}
\begin{equation*}
\ddot{\gamma}(t) = \P^{\frac{1}{2}}\left(\frac{d}{dt}\big(\exp(t\L)\L\right)\P^{\frac{1}{2}}
= \P^{\frac{1}2}\exp(t\L)\L^{2}\P^{\frac{1}{2}}\,,
\end{equation*}
and
\begin{equation*}
\gamma(t)^{-1} =\left(\P^{\frac{1}{2}}\exp(t\L)\P^{\frac{1}{2}}\right)^{-1} = \P^{-\frac{1}{2}}\exp(-t\L)\P^{-\frac{1}{2}}\,.
\end{equation*}
Therefore
\begin{eqnarray*}
\dot{\gamma}(t)\gamma(t)^{-1}\dot{\gamma}(t) & = & \left(\P^{\frac{1}{2}}\exp(t\L)\L\P^{\frac{1}{2}}\right) \left(\P^{-\frac{1}{2}}\exp(-t\L)\P^{-\frac{1}{2}}\right) \left(\P^{\frac{1}{2}}\exp(t\L)\L\P^{\frac{1}{2}}\right) \\
& = & \P^{\frac{1}{2}}\exp(t\L)\L\exp(-t\L)\exp(t\L)\L\P^{\frac{1}{2}} \\
& = & \P^{\frac{1}{2}}\exp(t\L)\L^{2}\P^{\frac{1}{2}} = \ddot{\gamma}(t)\,,
\end{eqnarray*}
where we used $\exp(-t\L)\exp(t\L) = \Id_{n}$ and the fact that $\L$ commutes with $\exp(t\L)$\,.

\end{remark}

\noindent We now construct a Riemannian metric on $\mathscr{P}_{\J}$ by using the metric $\langle\cdot\,, \cdot\rangle$ on $\mathscr{P}$. Let $\Phi_{\J}: \mathscr{P}_{\J} \to \mathscr{P}$ be the map $\Phi_{\J}(\X) = \J\X$. Let $\omega = \Phi^{*}_{\J}(\langle\cdot\,, \cdot\rangle)$ be the pull-back of the Riemannian metric to $\mathscr{P}_{\J}$. By definition, for all $\P \in \mathscr{P}_{\J}$, we have
\begin{equation*}
\omega_{\P}(\U\,, \V) = \langle \J\U\,, \J\V\rangle_{\J\P}\,, \qquad \left(\U\,, \V \in \mathfrak{p}_{\J}\right)\,. 
\end{equation*}
In particular, we get
\begin{eqnarray*}
\omega_{\P}(\U\,, \V) & = & \langle\J\U\,, \J\V\rangle_{\J\P} = \widetilde{\tr}\left((\J\P)^{-1}\J\U(\J\P)^{-1}\J\V\right) \\
& = & \widetilde{\tr}(\P^{-1}\J^{-1}\J\U\P^{-1}\J^{-1}\J\V) = \widetilde{\tr}(\P^{-1}\U\P^{-1}\V)\,,
\end{eqnarray*}
i.e.
\begin{equation*}
\omega_{\P}(\U\,, \V) = \widetilde{\tr}(\P^{-1}\U\P^{-1}\V)\,, \qquad \qquad \left(\P \in \mathscr{P}_{\J}\,, \U\,, \V \in \mathfrak{p}_{\J}\right)\,.
\end{equation*}

\noindent The group $\GL(n)$ acts on $\mathscr{P}_{\J}$ by $g \cdot \P = g\P g^{\sharp}$ (see Proposition \ref{PropositionPropertiesSRJ}). The form $\omega$ is $\GL(n)$-invariant, i.e.
\begin{equation*}
\omega_{g\X g^{\sharp}}\left(g\U g^{\sharp}\,, g\V g^{\sharp}\right) = \omega_{\X}(\U\,, \V) \qquad \qquad \left(\X \in \mathscr{P}_{\J}\,, \U\,, \V \in \mathfrak{p}_{\J}\right)\,.
\end{equation*}
Indeed,
\begin{eqnarray*}
& &\omega_{g\X g^{\sharp}}\left(g\U g^{\sharp}\,, g\V g^{\sharp}\right)
= \widetilde{\tr}\left((g\X g^{\sharp})^{-1}g\U g^{\sharp}(g\X g^{\sharp})^{-1}g\V g^{\sharp}\right) \\ 
& = & \widetilde{\tr}((g^{\sharp})^{-1}\X^{-1}g^{-1}g\U g^{\sharp}(g^{\sharp})^{-1}\X^{-1} g^{-1}g\V g^{\sharp}) = \widetilde{\tr}((g^{\sharp})^{-1}\X^{-1}\U\X^{-1}\V g^{\sharp}) \\ 
& = & \widetilde{\tr}(\X^{-1}\U\X^{-1}\V) = \omega_{\X}(\U\,, \V)\,.
\end{eqnarray*}

\noindent Moreover, the form is by construction positive. Indeed, $\langle\cdot,, \cdot\rangle$ is positive and for all $\U \in \mathfrak{p}_{\J}$ non-zero and $\P \in \mathscr{P}_{\J}$, we have
\begin{equation*}
\omega_{\P}(\U\,, \U) = \langle\J\U\,, \J\U\rangle_{\J\P} > 0\,.
\end{equation*}

\noindent We now give an analogue of Equation \eqref{GeodesicInP} on the cone $\mathscr{P}_{\J}$.

\begin{theo}
Let $\A,\B\in \mathscr{P}_{\J}$. The map $\gamma: \left[0\,, 1\right] \to \mathscr{P}_{\J}$ given by
\begin{equation}
\gamma(t) := \A^{\frac{1}{2}}_{\J} \bullet \left(\A^{-\frac{1}{2}}_{\J} \bullet \B \bullet \A^{-\frac{1}{2}}_{\J}\right)^{t}_{\J} \bullet \A^{\frac{1}{2}}_{\J}\,, \qquad \left(t \in \left[0\,, 1\right]\right)
\label{GeodesicInPJ}
\end{equation}
is the equation of the geodesic between $\A$ and $\B$ in $\mathscr{P}_{\J}$ with respect to $\omega$ (in particular, $\gamma(0) = \A$ and $\gamma(1) = \B$)\,.

\label{TheoremGeodesicsPJ}

\end{theo}

\noindent Before proving Theorem \ref{TheoremGeodesicsPJ}, we recall the following lemma\,.

\begin{lemma}

Let $\left(\M\,, g\right)$ and $\left(\N\,, h\right)$ be two smooth Riemannian manifolds, and let $\Phi: \M \to \N$ be an isometry (i.e. a diffeomorphism such that $g = \Phi^{*}h$). Then $\Phi$ maps geodesics in $\M$ to geodesics in $\N$. In other words, $\gamma$ is a geodesic in $\M$ if and only if $\Phi \circ \gamma$ is a geodesic in $\N$\,.

\label{LemmaRiemannianManifolds}

\end{lemma}

\begin{proof}

This is a standard result in Riemannian geometry; see \cite[Proposition~5.4]{LEE}\,.

\end{proof}

\begin{proof}[Proof of Theorem \ref{TheoremGeodesicsPJ}]

Let $\delta$ be the equation of the geodesic between $\J\A$ and $\J\B$ in $\mathscr{P}$. In particular, using Equation \eqref{GeodesicInP}, we have
\begin{equation*}
\delta(t) = (\J\A)^{\frac{1}{2}}\left((\J\A)^{-\frac{1}{2}}(\J\B)(\J\A)^{-\frac{1}{2}}\right)^{t} (\J\A)^{\frac{1}{2}}\,.
\end{equation*}
According to Lemma \ref{LemmaRiemannianManifolds}, it is enough to prove that $\gamma(t) = \J\delta(t)$ for all $t \in \left[0\,, 1\right]$. From Proposition \ref{PropositionSquareRoot}, we have
\begin{equation*}
\X^{t}_{\J} = \J(\J\X)^{t}\,, \qquad \quad \left(\X \in \mathscr{P}_{\J}\,, t \in \mathbb{R}\right)\,.
\end{equation*}
Therefore, $\A^{\frac{1}{2}}_{\J} = \J(\J\A)^{\frac{1}{2}}$ and $\A^{-\frac{1}{2}}_{\J} = \J(\J\A)^{-\frac{1}{2}}$. Moreover,
\begin{eqnarray*}
\A^{-\frac{1}{2}}_{\J} \bullet \B \bullet \A^{-\frac{1}{2}}_{\J} & = & \left(\J(\J\A)^{-\frac{1}{2}}\right)(\J\B\J)\left(\J(\J\A)^{-\frac{1}{2}}\right) = \J(\J\A)^{-\frac{1}{2}}\left(\J\B\J\J\right)(\J\A)^{-\frac{1}{2}} \\
& = & \J(\J\A)^{-\frac{1}{2}}(\J\B)(\J\A)^{-\frac{1}{2}}\,.
\end{eqnarray*}
Then
\begin{eqnarray*}
\gamma(t) & = & \A^{\frac{1}{2}}_{\J}\bullet\left(\A^{-\frac{1}{2}}_{\J}\bullet\B\bullet\A^{-\frac{1}{2}}_{\J}\right)^{t}_{\J}\bullet\A^{\frac{1}{2}}_{\J} \\ 
& = & \J(\J\A)^{\frac{1}{2}}\J^{2}\left(\J^{2}(\J\A)^{-\frac{1}{2}}(\J\B)(\J\A)^{-\frac{1}{2}}\right)^{t}\J\left(\J(\J\A)^{\frac{1}{2}}\right) \\ 
& = & \J(\J\A)^{\frac{1}{2}}\left((\J\A)^{-\frac{1}{2}}(\J\B)(\J\A)^{-\frac{1}{2}}\right)^{t}(\J\A)^{\frac{1}{2}} \\
& = & \J\delta(t)\,,
\end{eqnarray*}
and the theorem follows\,.

\end{proof}

\begin{remark}

In particular, one can see that the map $\gamma$ given in Equation \eqref{GeodesicInPJ} satisfies the ODE given in Equation \eqref{ODEP}. Indeed, using that $\gamma(t) = \J\delta(t)$ and that $\ddot{\delta}(t) = \dot{\delta}(t)\delta(t)^{-1}\dot{\delta}(t)$, we get that 
$\dot{\gamma}(t) = \J\dot{\delta}(t)$, $\gamma(t)^{-1} = (\J\delta(t))^{-1} = \delta(t)^{-1}\J$. Then

\begin{equation*}
\dot{\gamma}(t)\gamma(t)^{-1}\dot{\gamma}(t) = \left(\J\dot{\delta}(t)\right)\left(\delta(t)^{-1}\J\right)\left(\J\dot{\delta}(t)\right) = \J\dot{\delta}(t)\delta(t)^{-1}\dot{\delta}(t) = \J\ddot{\delta}(t) = \ddot{\gamma}(t)\,,
\end{equation*}
which is exactly the ODE given in Equation \eqref{ODEP}\,.

\end{remark}

\section{Geometric mean on the cone $\mathscr{P}_{\J}$}

\label{SectionFive}

We start this section by recalling the main results of \cite{KuboAndo1980} concerning the geometric mean on $\mathscr{P}$ (note that the results of \cite{KuboAndo1980} are for $\mathbb{D} \in \left\{\mathbb{R}\,, \mathbb{C}\right\}$; an analogue in the quaternionic case can be found in Appendix \ref{AppendixA}). For two matrices $\A\,, \B \in \mathscr{P}$ and $t \in \left[0\,, 1\right]$, we denote by $\A \sharp_{t} \B$ the matrix in $\Mat(n \times n)$ given by
\begin{equation*}
\A \sharp_{t} \B = \A^{\frac{1}{2}}\left(\A^{-\frac{1}{2}}\B\A^{-\frac{1}{2}}\right)^{t}\A^{\frac{1}{2}}\,,
\end{equation*}
and let $\A \sharp \B = \A \sharp_{\frac{1}{2}} \B$. As explained in Equation \eqref{GeodesicInP}, $\A \sharp_{t} \B \in \mathscr{P}$ for all $t \in \left[0\,, 1\right]$. 

\begin{remark}
    
Here are some nice properties satisfied by $\sharp_{t}$ on $\mathscr{P}$\,:
\begin{enumerate}
\item For all $\A\,, \B \in \mathscr{P}$ and $t \in \left(0\,, 1\right)$, $\A \sharp_{t} \A = \A$ and $\A \sharp_{t} \B = \A$ if and only if $\A = \B$\,.
\item For all $\A\,, \B \in \mathscr{P}$, $\A \sharp \B$ is the unique solution of Riccati's equation
\begin{equation*}
\X\A^{-1}\X = \B\,,
\end{equation*}
\item For all $\A\,, \B \in \mathscr{P}$ and $\C \in \GL(n)$, we have
\begin{equation*}
\left(\C^{*}\A\C\right) \sharp \left(\C^{*}\B\C\right) = \C^{*}\left(\A \sharp \B\right)\C\,.
\end{equation*}
\item For all $\A_{1}\,, \A_{2}\,, \B_{1}\,, \B_{2} \in \mathscr{P}$ such that $\A_{1} \preceq \A_{2}$ and $\B_{1} \preceq \B_{2}$, we get
\begin{equation*}
\A_{1} \sharp \B_{1} \preceq \A_{2} \sharp \B_{2}\,.
\end{equation*}
\item For all $\A\,, \B \in \mathscr{P}$, then $\left(\A \sharp \B\right)^{-1} = \A^{-1} \sharp \B^{-1}$\,.
\end{enumerate}

\label{RemarkSharp}

\end{remark}

\noindent The goal of this section is to give an analogue of the geometric mean on our cone $\mathscr{P}_{\J}$. We start with a definition\,.

\begin{definition}
    
Let $\A\,, \B \in \mathscr{P}_{\J}$. For all $t \in \left[0\,, 1\right]$, we denote by $\A \sharp^{\J}_{t} \B$ the element in $\Mat(n \times n)$ given by
\begin{equation*}
\A \sharp^{\J}_{t} \B := \A^{\frac{1}{2}}_{\J}\bullet\left(\A^{-\frac{1}{2}}_{\J}\bullet\B\bullet\A^{-\frac{1}{2}}_{\J}\right)^{t}_{\J}\bullet\A^{\frac{1}{2}}_{\J}\,,
\end{equation*}
and let $\A \sharp^{\J} \B = \A \sharp^{\frac{1}{2}}_{\J} \B$\,.

\label{DefinitionJGeometricMean}

\end{definition}

\begin{remark}

\begin{enumerate}
\item As explained in Theorem \ref{TheoremGeodesicsPJ}, for all $\A\,, \B \in \mathscr{P}_{\J}$, we get $\A \sharp^{\J}_{t} \B \in \mathscr{P}_{\J}$ for all $t \in \left[0\,, 1\right]$\,.
\item It follows from the proof of Theorem \ref{TheoremGeodesicsPJ} that for all $\A\,, \B \in \mathscr{P}_{\J}$, we have
\begin{equation*}
\A \sharp^{\J}_{t} \B = \Phi^{-1}_{\J}\left(\Phi_{\J}(\A) \sharp_{t} \Phi_{\J}(\B)\right)\,, \qquad \qquad \left(t \in \left[0\,, 1\right]\right)\,.
\end{equation*}
\end{enumerate}

\label{RemarkGeometricMeanJ}

\end{remark}

\begin{proposition}[Riccati Equation]
Let $\A\,, \B \in \mathscr{P}_{\J}$. Then $\A \sharp^{\J} \B$ is the unique solution on $\mathscr{P}_{\J}$ of the equation
\begin{equation}
\label{RicattiEquation}
\X\A^{-1}\X = \B\,.
\end{equation}

\label{PropositionRicatti}

\end{proposition}

\begin{proof}

Firstly, we establish that $\A \sharp^{\J} \B$ is indeed a solution to \eqref{RicattiEquation}. Using Proposition \ref{PropositionSquareRoot}, we get that $\A^{1}_{\J}\J\A^{-1}_{\J} = \J$ (i.e. $\A^{-1} = \J\A^{-1}_{\J}\J$) and $\A^{-1}_{\J} = \A^{-\frac{1}{2}}_{\J}\J\A^{-\frac{1}{2}}_{\J}$. Therefore
\begin{equation*}
\X\A^{-1}\X=\X\J\A_{\J}^{-1}\J\X=\X\J \A^{-\frac{1}{2}}_{\J}\J\A^{-\frac{1}{2}}_{\J}\J\X\,.
\end{equation*}
Thus
\begin{eqnarray*}
& &\left(\A \sharp^{\J} \B\right)\A^{-1}\left(\A \sharp^{\J} \B\right) = \left(\A \sharp^{\J} \B\right)\J\A^{-\frac{1}{2}}_{\J}\J\A^{-\frac{1}{2}}_{\J}\J\left(\A \sharp^{\J} \B\right) \\
& = & \A_{\J}^{\frac{1}{2}}\J\left(\A^{-\frac{1}{2}}_{\J}\J\B\J\A^{-\frac{1}{2}}_{\J}\right)^{\frac{1}{2}}_{\J}\J\left(\A^{-\frac{1}{2}}_{\J}\J\B\J \A^{-\frac{1}{2}}_{\J}\right)^{\frac{1}{2}}_{\J}\J\A^{\frac{1}{2}}_{\J} = \A^{\frac{1}{2}}_{\J}\J\left(\A^{-\frac{1}{2}}_{\J}\J\B\J\A^{-\frac{1}{2}}_{\J}\right)\J\A^{\frac{1}{2}}_{\J} \\
& = & \left(\A^{\frac{1}{2}}_{\J}\J\A^{-\frac{1}{2}}_{\J}\right)\J\B\J\left(\A^{-\frac{1}{2}}_{\J}\J\A^{\frac{1}{2}}_{\J}\right) = \J^{2}\B\J^{2} = \B\,,
\end{eqnarray*}
where the last two equalities follows from Proposition \ref{PropositionSquareRoot}\,. 

\noindent We now prove the unicity of the solution. Assume that $\X$ is such that
\begin{equation*}
\X\J \A^{-\frac{1}{2}}_{\J}\J \A^{-\frac{1}{2}}_{\J}\J\X=\B\,.
\end{equation*}
Then
\begin{equation*}
\left(\A^{-\frac{1}{2}}_{\J}\J\X\J\A^{-\frac{1}{2}}_{\J}\right)\J \left(\A^{-\frac{1}{2}}_{\J}\J\X\J\A^{-\frac{1}{2}}_{\J}\right) = \A^{-\frac{1}{2}}_{\J}\J\B\J\A^{-\frac{1}{2}}_{\J}\,.
\end{equation*}
Using Proposition \ref{PropositionSquareRoot} part (4), we have that
\begin{equation*}
\left(\A^{-\frac{1}{2}}_{\J}\J\X\J\A^{-\frac{1}{2}}_{\J}\right)^{2}_{\J} = \A^{-\frac{1}{2}}_{\J}\J\B\J\A^{-\frac{1}{2}}_{\J}\,.
\end{equation*}
Now, we take the $\J$-square root on both sides to obtain
\begin{equation*}
\A^{-\frac{1}{2}}_{\J}\J\X\J \A^{-\frac{1}{2}}_{\J}=(\A^{-\frac{1}{2}}_{\J}\J\B\J \A^{-\frac{1}{2}}_{\J})_{\J}^{\frac{1}{2}}\,.
\end{equation*}
To solve for $\X$ while staying inside $\mathscr{P}_{\J}$, we apply the congruence map $\Y \to \C\Y\C^{\sharp}$ with $\C = \A^{\frac{1}{2}}_{\J}\J$, which preserves $\mathscr{P}_{\J}$. Finally
\begin{equation*}
\J^{2}\X\J^{2}= \A \sharp^{\J} \B\,,
\end{equation*}
i.e. $\X = \A \sharp^{\J} \B$\,.
\end{proof}

\noindent We now prove some properties of $\sharp^{\J}$\,.

\begin{lemma}

For all $\A\,, \B \in \mathscr{P}_{\J}$, we obtain
\begin{enumerate}
\item $\A \sharp^{\J} \B = \B \sharp^{\J} \A$\,,
\item $\left(\A \sharp^{\J} \B\right)^{-1} = \A^{-1} \sharp^{\J} \B^{-1}$\,.
\end{enumerate}
    
\end{lemma}

\begin{proof}

\begin{enumerate}
\item Taking inverses in \eqref{RicattiEquation} yields
\begin{equation*}
\X^{-1}\A\X^{-1} = \B^{-1}\,.
\end{equation*}
Multiplying on the left and right by $\B$ gives
\begin{equation*}
\left(\B\X^{-1}\B\right)\B^{-1}\left(\B\X^{-1}\B\right) = \A\,.
\end{equation*}
Set $\Y := \B\X^{-1}\B$. Since $\mathscr{P}_{\J}$ is stable under inversion and congruence, we have $\Y \in \mathscr{P}_{\J}$. Hence $\Y$ is a solution in $\mathscr{P}_{\J}$ of the Riccati equation $\Z\B^{-1}\Z = \A$. By uniqueness of the solution in $\mathscr{P}_{\J}$, it follows that
\begin{equation*}
\B\sharp^{\J}\A = \Y = \B\X^{-1}\B\,.
\end{equation*}
Applying the same argument with $\A$ and $\B$ interchanged shows that $\X = \A\left(\B \sharp^{\J}\A\right)^{-1}\A$, and therefore
\begin{equation*}
\A \sharp^{\J} \B = \B \sharp^{\J} \A\,.
\end{equation*}
\item Again starting from Equation \eqref{RicattiEquation}, taking inverses gives
\begin{equation*}
\X^{-1}\A\X^{-1} = \B^{-1}\,.
\end{equation*}
Equivalently,
\begin{equation*}
\X^{-1}\left(\A^{-1}\right)^{-1}\X^{-1} = \B^{-1}\,.
\end{equation*}
Thus $\X^{-1} \in \mathscr{P}_{\J}$ is a solution of the Riccati equation associated with the pair $\left(\A^{-1}\,, \B^{-1}\right)$. By uniqueness of the solution in $\mathscr{P}_{\J}$, we conclude that
\begin{equation*}
\A^{-1} \sharp^{\J} \B^{-1} = \X^{-1} = \left(\A \sharp^{\J} \B\right)^{-1}\,.
\end{equation*}
\end{enumerate}
\noindent This completes the proof\,.

\end{proof}

\begin{lemma}

Let $\A, \B \in \mathscr{P}_{\J}$. Then
\begin{equation*}
\A \sharp^{\J} \B = \max\left\{\X \in \mathscr{P}_{\J}\,, \begin{bmatrix} \J\A & \J\X \\ \J\X & \J\B \end{bmatrix} \succeq 0\right\}\,,
\end{equation*}
where the maximum is understood to be with respect to the $\J$-Loewner order $\preceq_{\J}$\,.
    
\end{lemma}

\begin{proof}

Notice that the Schur complement of 
\begin{equation*}
\begin{bmatrix} \J\A & \J\X \\ \J\X & \J\B \end{bmatrix}
\end{equation*}
is $\J\B-\left(\J\X\right)\left(\J\A\right)^{-1}\left(\J\X\right) = \J\left(\B - \X\A^{-1}\X\right)$. When $\X = \A \sharp^{\J} \B$, the Schur complement is zero. Since $\J\A > 0$, it holds that 
\begin{equation*}
\begin{bmatrix} \J\A & \J\X \\ \J\X & \J\B \end{bmatrix} \succeq 0\,.
\end{equation*}
On the other hand, if $\X\A^{-1}\X \preceq_{\J} \B$, then 
\begin{equation*}
\J\B \succeq \left(\J\X\right)\left(\J\A\right)^{-1}\left(\J\X\right)\,.
\end{equation*}
This implies that 
\begin{equation*}
\left(\J\A\right)^{-\frac{1}{2}}\left(\J\B\right)\left(\J\A\right)^{-\frac{1}{2}} \succeq \left(\left(\J\A\right)^{-\frac{1}{2}}\left(\J\X\right)\left(\J\A\right)^{-\frac{1}{2}}\right)^{2}\,.
\end{equation*}
By the operator monotonicity of the square root and multiplying on both sides by $\left(\J\A\right)^{\frac{1}{2}}$, this shows that $\J\X \preceq \left(\J\A\right) \sharp \left(\J\B\right) $. Thus $\X \preceq_{\J} \A \sharp^{\J} \B$\,.

\end{proof}

\begin{proposition}
\label{PropertiesGeometric}
Let $t \in \left(0\,, 1\right)$\,.
\begin{enumerate}
\item For all $\A \in \mathscr{P}_{\J}$, we have $\A \sharp^{\J}_{t} \A = \A$. Moreover, we have $\A \sharp^{\J}_{t} \B = \A$ if and only if $\A = \B$\,.
\item For all $\A\,, \B \in \mathscr{P}_{\J}$ and $\alpha, \beta >0$, we have
\begin{equation*}
\left(\alpha\A\right) \sharp^{\J}_{t} \left(\beta\B\right) = \alpha^{1-t}\beta^{t} \A \sharp^{\J}_{t} \B\,.
\end{equation*}
\item For all $\A\,, \B \in \mathscr{P}_{\J}$, we have
\begin{equation*}
\A \sharp^{\J}_{t} \B = \B \sharp_{1-t}^{\J} \A\,.
\end{equation*}
\item If $\A\,, \B\,, \C\,, \D \in \mathscr{P}_{\J}$ are such that $\A \preceq_{\J} \C$ and $\B \preceq_{\J} \D$, then
\begin{equation*}
\A \sharp^{\J}_{t} \B \preceq_{\J} \C \sharp^{\J}_{t} \D\,.
\end{equation*}
\end{enumerate}

\end{proposition}

\begin{proof}

Recall that the weighted $\J$-geometric mean is defined by 
\begin{equation*}
\A \sharp^{\J}_{t} \B := \J^{-1}\left(\left(\J\A\right) \sharp_{t} \left(\J\B\right)\right)\,, \qquad \left(\A\,, \B \in \mathscr{P}_{\J}\right)\,,
\end{equation*}
where $\sharp_{t}$ denotes the usual weighted geometric mean on $\mathscr{P}$. Since $\J^{-1} = \J$, this is equivalent to
\begin{equation}
\label{EquationPullBackProof}
\J\left(\A \sharp^{\J}_{t} \B\right) = \left(\J\A\right) \sharp_{t} \left(\J\B\right)\,.
\end{equation}
We prove each item by transporting the corresponding standard property from
$\mathscr{P}$ to $\mathscr{P}_{\J}$ via \eqref{EquationPullBackProof}\,.

\begin{enumerate}
\item Let $\A\in\mathscr{P}_{\J}$. Using Equation \eqref{EquationPullBackProof} and Remark \ref{RemarkSharp}, we get that
\begin{equation*}
\J\left(\A \sharp^{\J}_{t} \A\right) = \left(\J\A\right) \sharp_{t} \left(\J\A\right) = \J\A\,,
\end{equation*}
hence $\A \sharp^{\J}_{t} \A = \A$\,.

\noindent Now assume $\A \sharp^{\J}_{t} \B = \A$. Multiplying by $\J$ and using \eqref{EquationPullBackProof} gives
\begin{equation*}
\left(\J\A\right) \sharp_{t} \left(\J\B\right) = \J\A\,.
\end{equation*}
Again, it follows from Remark \ref{RemarkSharp} that $\J\A = \J\B$, i.e. $\A = \B$\,.
\item Let $\alpha\,, \beta > 0$ and $\A\,, \B \in \mathscr{P}_{\J}$. First note that
\begin{equation*}
\left(\alpha\A\right)^{\frac{1}{2}}_{\J} = \alpha^{\frac{1}{2}}\A^{\frac{1}{2}}_{\J}\,, \qquad \left(\alpha\A\right)^{-\frac{1}{2}}_{\J} =\alpha^{-\frac{1}{2}}\A^{-\frac{1}{2}}_{\J}\,.
\label{EquationPowerProof}
\end{equation*}
We get
\begin{eqnarray*}
\left(\alpha\A\right) \sharp^{\J}_{t} \left(\beta\B\right) & = &\left(\alpha\A\right)^{\frac{1}{2}}_{\J}\J\left(\left(\alpha\A\right)^{-\frac{1}{2}}_{\J}\J\left(\beta\B\right)\J\left(\alpha\A\right)^{-\frac{1}{2}}_{\J}\right)^{t}_{\J}\J\left(\alpha\A\right)^{\frac{1}{2}}_{\J} \\
& = & \alpha^{\frac{1}{2}}\A^{\frac{1}{2}}_{\J}\J\left(\alpha^{-\frac{1}{2}}\A^{-\frac{1}{2}}_{\J}\left(\beta\J\B\right)\alpha^{-\frac{1}{2}}\A^{-\frac{1}{2}}_{\J}\right)^{t}_{\J}\J\alpha^{\frac{1}{2}}\A^{\frac{1}{2}}_{\J} \\
& = & \alpha^{\frac{1}{2}}\A^{\frac{1}{2}}_{\J}\J\left(\tfrac{\beta}{\alpha}\left(\A^{-\frac{1}{2}}_{\J}\J\B\J\A^{-\frac{1}{2}}_{\J}\right)\right)^{t}_{\J}\J\alpha^{\frac{1}{2}}\A^{\frac{1}{2}}_{\J}\,.
\end{eqnarray*}
Using that $\left(\lambda \A\right)^{t}_{\J} = \lambda^{t}\A^{t}_{\J}$, we get
\begin{equation*}
\left(\tfrac{\beta}{\alpha}\left(\A^{-\frac{1}{2}}_{\J}\J\B\J\A^{-\frac{1}{2}}_{\J}\right)\right)^{t}_{\J} = \left(\tfrac{\beta}{\alpha}\right)^{t}\left(\A^{-\frac{1}{2}}_{\J}\J\B\J\A^{-\frac{1}{2}}_{\J}\right)^{t}_{\J}\,.
\end{equation*}
Therefore,
\begin{equation*}
\left(\alpha\A\right) \sharp^{\J}_{t} \left(\beta\B\right) = \alpha^{\frac{1}{2}}\left(\tfrac{\beta}{\alpha}\right)^{t}\alpha^{\frac{1}{2}}
\A^{\frac{1}{2}}_{\J}\J\left(\A^{-\frac{1}{2}}_{\J}\J\B\J\A^{-\frac{1}{2}}_{\J}\right)^{t}_{\J}\J\A^{\frac{1}{2}}_{\J} = \alpha^{1-t}\beta^{t}\left(\A \sharp^{\J}_{t} \B\right)\,,
\end{equation*}
which proves (2)\,.
\item Consider the curve
\begin{equation*}
\gamma(t) := \A \sharp^{\J}_{t} \B\,, \qquad \qquad \left(t \in \left[0\,, 1\right]\right)\,.
\end{equation*}
By construction, $\gamma(0) = \A$ and $\gamma(1) = \B$. Let $\widetilde{\gamma}(t) := \gamma(1-t)$. Then $\widetilde{\gamma}(0) = \B$ and $\widetilde{\gamma}(1) = \A$.
Moreover, $\widetilde\gamma$ satisfies the same geodesic equation as $\gamma$. By the uniqueness of the geodesic in $\mathscr{P}_{\J}$ joining two given endpoints, it follows that $\widetilde{\gamma}$ coincides with the geodesic $t \mapsto \B \sharp^{\J}_{t} \A$. Hence, for every $t \in \left[0\,, 1\right]$,
\begin{equation*}
\B \sharp^{\J}_{t} \A = \widetilde{\gamma}(t) = \gamma(1-t) = \A \sharp^{1-t}_{\J} \B\,,
\end{equation*}
which is equivalent (after replacing $t$ by $1-t$) to
\begin{equation*}
\A \sharp^{\J}_{t} \B = \B \sharp^{1-t}_{\J} \A\,.
\end{equation*}
This proves (3)\,.
\item Assume $\A \preceq_{\J} \C$ and $\B \preceq_{\J} \D$. By definition of $\preceq_{\J}$, this is equivalent to $\J\A \preceq \J\C$ and $\J\B\preceq \J\D$ in the usual Loewner order on $\mathscr{P}$. Using Remark \ref{RemarkSharp}, it follows that
\begin{equation*}
\left(\J\A\right) \sharp_{t} \left(\J\B\right) \preceq \left(\J\C\right) \sharp_{t} \left(\J\D\right)\,.
\end{equation*}
Therefore
\begin{equation*}
\J\left(\A \sharp^{\J}_{t} \B\right) \preceq \J\left(\C \sharp^{\J}_{t} \D\right)\,,
\end{equation*}
which is exactly $\A \sharp^{\J}_{t} \B \preceq_{\J} \C \sharp^{\J}_{t} \D$\,.
\end{enumerate}

\end{proof}

\begin{proposition}

\begin{enumerate}
\item For all $g \in \K_{\J}$, $t\in \left[0\,, 1\right]$ and $\A\,, \B \in \mathscr{P}_{\J}$, we have
\begin{equation*}
\left(g\A g^{\sharp}\right) \sharp^{\J}_{t} \left(g\B g^{\sharp}\right) = g\left(\A \sharp^{\J}_{t} \B\right)g^{\sharp}\,.
\end{equation*}
\item For all $\A\,, \B\,, \C\,, \D \in \mathscr{P}_{\J}$ and $t\,, s \in \left[0\,, 1\right]$, we have
\begin{equation*}
\left(1-s\right)\A \sharp^{\J}_{t} \C + s \B \sharp^{\J}_{t} \D \preceq_{\J} \left(\left(1-s\right)\A + s \B\right) \sharp^{\J}_{t} \left(\left(1-s\right)\C + s \D\right)\,.
\end{equation*}
\item For all $\A\,, \B \in \mathscr{P}_{\J}$ and $t\,, s\,, u \in \left[0\,, 1\right]$,
\begin{equation*}
\left(\A \sharp^{\J}_{t} \B\right) \sharp_{u}^{\J} \left(\A \sharp_{s}^{\J} \B\right) = \A \sharp_{(1-u)t+us}^{\J} \B\,.
\end{equation*}
\item For all $\A\,, \B \in \mathscr{P}_{\J}$ and $t \in \left[0\,, 1\right]$, we get
\begin{equation*}
[(1-t)\A_{\J}^{-1}+t\B_{\J}^{-1}]^{-1}\preceq_{\J} \A \sharp_{t}^{\J} \B \preceq_{\J} (1-t) \A + t \B\,.
\end{equation*}

\end{enumerate}

\end{proposition}

\begin{proof}

\begin{enumerate}
\item Let $g \in \K_{\J}$ and let $\widetilde{\A} = g\A g^{\sharp}$ and $\widetilde{\B} = g\B g^{\sharp}$.Using Proposition \ref{PropositionPropertiesSRJ}, for every $\X \in \mathscr{P}_{\J}$ and $t \in \mathbb{R}$, we have $(g\X g^{\sharp})^{t}_{\J} = g\X^{t}_{\J}g^{\sharp}$. In particular,
\begin{equation*}
\widetilde{\A}^{\pm\frac{1}{2}}_{\J} = g\A^{\pm\frac{1}{2}}_{\J}g^{\sharp}\,.
\end{equation*}
By definition of $\sharp^{\J}_{t}$ (see \ref{DefinitionJGeometricMean}), we have
\begin{equation*}
\widetilde{\A} \sharp^{\J}_{t} \widetilde{\B} =
\widetilde{\A}^{\frac{1}{2}}_{\J} \left(
\widetilde{\A}^{-\frac{1}{2}}_{\J}\widetilde{\B}
\widetilde{\A}^{-\frac{1}{2}}_{\J}\right)^{t}_{\J} \widetilde{\A}^{\frac{1}{2}}_{\J}\,.
\end{equation*}
Using that $g^{\sharp}g = \Id_{n}$, we obtain
\begin{equation*}
\widetilde{\A}^{-\frac{1}{2}}_{\J}\widetilde{\B}
\widetilde{\A}^{-\frac{1}{2}}_{\J} = \left(g\A^{-\frac{1}{2}}_{\J}g^{\sharp}\right)\left(g\B g^{\sharp}\right)
\left(g\A^{-\frac{1}{2}}_{\J}g^{\sharp}\right) = g\left(\A^{-\frac{1}{2}}_{\J}\B\A^{-\frac{1}{2}}_{\J}\right)g^{\sharp}\,.
\end{equation*}
Finally, using that
\begin{equation*}
\left( \widetilde{\A}^{-\frac{1}{2}}_{\J}\widetilde{\B}
\widetilde{\A}^{-\frac{1}{2}}_{\J}\right)^{t}_{\J} =
g\left(\A^{-\frac{1}{2}}_{\J}\B\A^{-\frac{1}{2}}_{\J}\right)^{t}_{\J}g^{\sharp}\,,
\end{equation*}
it follows that
\begin{eqnarray*}
\widetilde{\A}\sharp^{\J}_{t}\widetilde{\B} & = & \left(g\A^{\frac{1}{2}}_{\J}g^{\sharp}\right) \left(
g(\A^{-\frac{1}{2}}_{\J}\B\A^{-\frac{1}{2}}_{\J})^{t}_{\J}g^{\sharp}
\right) \left(g\A^{\frac{1}{2}}_{\J}g^{\sharp}\right) \\
& = & g\A^{\frac{1}{2}}_{\J}\left(\A^{-\frac{1}{2}}_{\J}\B\A^{-\frac{1}{2}}_{\J}\right)^{t}_{\J} \A^{\frac{1}{2}}_{\J}g^{\sharp} \\
& = & g\left(\A \sharp^{\J}_{t} \B\right)g^{\sharp}\,.
\end{eqnarray*}
\item Using Remark \ref{RemarkGeometricMeanJ}, we get
\begin{equation*}
\J\left(\A \sharp^{\J}_{t} \C\right) = \left(\J\A\right) \sharp_{t} \left(\J\C\right)\,, \qquad \J\left(\B \sharp^{\J}_{t} \D\right)=\left(\J\B\right) \sharp_{t} \left(\J\D\right)\,.
\end{equation*}
Hence,
\begin{equation*}
\J\left(\left(1-s\right)\A \sharp^{\J}_{t} \C + s\B \sharp^{\J}_{t} \D\right) = (1-s)\J\A \sharp_{t} \J\C + s\J\B \sharp_{t} \J\D\,.
\end{equation*}
Similarly, we get
\begin{equation*}
\J\left(\left(\left(1-s\right)\A + s\B\right) \sharp^{\J}_{t} \left(\left(1-s\right)\C + s\D\right)\right) = \left(\left(1-s\right)\J\A + s\J\B\right) \sharp_{t} \left(\left(1-s\right)\J\C + s\J\D\right)\,.
\end{equation*}

\noindent Using the concavity of the geometric mean on $\mathscr{P}$ (see \cite{KuboAndo1980}), we have
\begin{equation*}
(1-s)\J\A \sharp_{t} \J\C + s\J\B \sharp_{t} \J\D
\preceq \left(\left(1-s\right)\J\A + s\J\B\right) \sharp_{t} \left(\left(1-s\right)\J\C\ +s\J\D\right)\,.
\end{equation*}
Combining the previous identities, we obtain
\begin{equation*}
\J\left(\left(1-s\right)\A \sharp^{\J}_{t} \C + s\B \sharp^{\J}_{t} \D\right) \preceq 
\J\left(\left(\left(1-s\right)\A + s\B\right) \sharp^{\J}_{t} \left(\left(1-s\right)\C + s\D\right)\right)\,,
\end{equation*}
i.e.
\begin{equation*}
(1-s)\A \sharp^{\J}_{t} \C + s\B \sharp^{\J}_{t} \D \preceq_{\J} 
\left(\left(1-s\right)\A + s\B\right) \sharp^{\J}_{t} \left(\left(1-s\right)\C + s\D\right)\,.
\end{equation*}
\item Follows from \cite{KuboAndo1980} using similar arguments as in $(2)$\,.
\item This follows from the corresponding inequality for positive definite matrices.
\end{enumerate}
\end{proof}

\begin{remark}

A remark is in order. Considering that for positive definite matrices $\X\,, \Y \in \mathscr{P}$ that commute, we have 
\begin{equation}
\X \sharp_{t} \Y = \X^{1-t} \Y^{t}\,, \qquad \qquad \left(t \in \left[0\,, 1\right]\right)\,.
\label{LastEquation}
\end{equation}
it is natural to ask whether two $\J$-positive hermitian matrices that commute satisfy such an equation. That is, if $\A\,, \B \in \mathscr{P}_{\J}$ such that $\A\B = \B\A$, is it true that
$\A \sharp^{\J}_{t} \B = \A^{1-t}_{\J}\B^{t}_{\J}$?

\noindent Let $\J = \diag(1\,, -1)$ (i.e. $p = q = 1$). One  easy to see that the matrices
\begin{equation*}
\A = \begin{bmatrix} 2 & 1 \\ -1 & -2 \end{bmatrix} \qquad \text{ and } \qquad \B = \begin{bmatrix} 3 & 1 \\ -1 & -1 \end{bmatrix}
\end{equation*}
are both $\J$-positive matrices in $\mathscr{P}_{\J}$. Moreover, $\A$ and $\B$ commute in $\Mat(2 \times 2)$. However, we get
\begin{equation*}
\A \sharp^{\J} \B - \A^{\frac{1}{2}}_{\J}\B^{\frac{1}{2}}_{\J} \approx \begin{bmatrix} 0.263207 & 0.768429 \\ -0.857469 & -2.50336 \end{bmatrix}\,.
\end{equation*}
This is one key property in which the geometric mean of $\J$-positive matrices differs from the geometric mean of positive definite matrices. 

\label{LastRemark}

\end{remark}

\noindent The right analogue of Equation \eqref{LastEquation} on $\mathscr{P}_{\J}$ is given in the following lemma and is consistent with the algebraic structure provided by $\bullet$\,.

\begin{lemma}

For all $\A\,, \B \in \mathscr{P}_{\J}$ such that $\A\bullet\B = \B\bullet\A$, we have that
\begin{equation*}
\A \sharp^{\J}_{t} \B = \A^{1-t}_{\J}\bullet\B^{t}_{\J}\,.
\end{equation*}

\end{lemma}

\begin{proof}

Let $\A\,,\B \in \mathscr{P}_{\J}$ such that that $\A\J\B = \B\J\A$. Let $\P := \J\A$ and $\Q := \J\B \in \mathscr{P}$. Since $\J^{2} = \Id_{n}$, the commutation relation implies
\begin{equation*}
\P\Q = \left(\J\A\right)\left(\J\B\right) = \J\left(\A\J\B\right) = \J\left(\B\J\A\right) = \left(\J\B\right)\left(\J\A\right) = \Q\P\,.
\end{equation*}
Hence $\P$ and $\Q$ commute. Therefore, it follows from Remark \ref{LastRemark} that for all $t \in \left[0\,, 1\right]$
\begin{equation*}
\P \sharp_{t} \Q = \P^{1-t} \Q^{t}\,.
\end{equation*}
Hence
\begin{equation*}
\A \sharp^{\J}_{t} \B = \J^{-1} \P^{1-t} \Q^{t} = \J^{-1} \left(\J\A\right)^{1-t} \left(\J\B\right)^{t}\,.
\end{equation*}
Inserting $\J^{2} = \Id_{n}$ between the two factors and using that $\X^{t}_{\J} := \J\left(\J\X\right)^{t}$, we obtain
\begin{equation*}
\J^{-1} \left(\J\A\right)^{1-t} \left(\J\B\right)^{t} = \left(\J\left(\J\A\right)^{1-t}\right)\J\left(\J\left(\J\B\right)^{t}\right) = \A^{1-t}_{\J}\J\B^{t}_{\J}\,,
\end{equation*}
and the lemma follows\,.

\end{proof}

\noindent Another important property satisfied by the weighted geometric means is the celebrated Ando-Hiai inequality \cite{AndoHiai1994}, that states that for $\A, \B \in \mathscr{P}$
\begin{equation}
\label{Ando-Hiai}
\A \sharp_{t} \B \preceq \Id_{n} \implies \A^{r} \sharp_{t} \B^{r} \preceq \Id_{n}\,, \qquad \left(r \geq 1\right)\,. 
\end{equation}

\begin{proposition}[$\J$-Ando-Hiai Inequality] 
\label{JAndoHiai}
For all $\A\,, \B \in \mathscr{P}_{\J}$ and $r\geq 1$, we have that
\begin{equation*}
\A \sharp^{\J}_{t} \B \preceq_{\J} \J \implies \A^{r}_{\J} \sharp^{\J}_{t} \B^{r}_{\J} \preceq_{\J} \J \,.
\end{equation*}

\end{proposition}

\begin{proof}

Notice that $\A \sharp_t^{\J} \B \preceq_{\J} \J$ if and only if $\J\A \sharp_t^{\J} \B \preceq \Id_{n}$. However,
\begin{equation*}
\J\A \sharp_t^{\J} \B= (\J\A)\sharp_t (\J \B)\preceq \Id_{n}\,,
\end{equation*}
implies
\begin{equation*}
\left(\J\A\right)^{r} \sharp_{t} \left(\J\B\right)^{r}\preceq \Id_{n}\,,
\end{equation*}
by \eqref{Ando-Hiai}. Therefore, $\left[\J(\A)^{r}_{\J}\right] \sharp_{t} \left[\J(\B)^{r}_{\J}\right] \preceq \Id_{n}$ or equivalently $\J(\A_{\J}^r\sharp_t \B_{\J}^r) \preceq \Id_{n}$. This yields the desired result in the $\J$-order\,.

\end{proof}

\noindent We conclude this article with a consequence of the $\J$-Ando-Hiai inequality, that is the $\J$-Furuta inequality. The proof is an adaptation to this setting of one of the proofs in \cite{FujiiKamei2006}.

\begin{proposition}[$\J$-Furuta Inequality]
For all $\A\,, \B \in \mathscr{P}_{\J}$, $p \geq 0$ and $r \geq 1$, we have that $0\prec_{\J}\B\preceq_{\J} \A$ implies
\begin{equation*}
\left(\A^{\frac{r}{2}}_{\J} \bullet \B^{p}_{\J} \bullet \A^{\frac{r}{2}}_{\J}\right)_{\J}^{\frac{r}{r+p}} \preceq_{\J} \A^{r}_{\J}\,.
\end{equation*}

\end{proposition}

\begin{proof}

Notice that $\B \preceq_{\J} \A$ implies $\A_{\J}^{-1} \preceq \B_{\J}^{-1}$, by Lemma \ref{LemmaInverse}. By part (4) of Proposition \ref{PropertiesGeometric}, we have that 
\begin{equation*}
\A_{\J}^{-1}\sharp^{\J}_{t} \B_{\J}^{q} \preceq_{\J} \B_{\J}^{t(1+q)-1} \,,
\end{equation*}
whenever $q \geq 0$ and $t\in \left[0\,, 1\right]$. In particular, pick $t = \frac{r}{p+r}$ and $q = \frac{p}{r}$. Then
\begin{equation*}
\A_{\J}^{-1}\sharp^{\J}_{\frac{r}{p+r}} \B_{\J}^{\frac{p}{r}} \preceq_{\J} \J \,.
\end{equation*}
Here, we used the fact that $\B^{0}_{\J} = \J$. Therefore, we can use Proposition \ref{JAndoHiai} and part (2) of Proposition \ref{PropositionSquareRoot} to obtain
\begin{equation*}
\A^{-r}_{\J}\sharp^{\J}_{\frac{r}{p+r}} \B^{p}_{\J} \preceq_{\J} \J \,.
\end{equation*}
At this point, we recall that
\begin{equation*}
\A^{-r}_{\J} \sharp^{\J}_{\frac{r}{p+r}} \B^{p}_{\J} = \A^{-\frac{r}{2}}_{\J} \bullet \left(\A^{\frac{r}{2}}_{\J} \bullet \B^{p}_{\J} \bullet \A^{\frac{r}{2}}_{\J}\right)^{\frac{r}{p+r}}_{\J}\bullet\A^{-\frac{r}{2}}_{\J}\,.
\end{equation*}
Thus, 
\begin{equation*}
\A^{-\frac{r}{2}}_{\J}\bullet\left(\A^{\frac{r}{2}}_{\J}\bullet\B_{\J}^{p}\bullet\A^{\frac{r}{2}}_{\J}\right)^{\frac{r}{p+r}}_{\J}\bullet\A^{-\frac{r}{2}}_{\J} \preceq_{\J} J\,.
\end{equation*}
Since $\A^{\frac{r}{2}}_{\J}$ is $\J$-Hermitian,
\begin{equation*}
\A^{\frac{r}{2}}_{\J}\A^{-\frac{r}{2}}_{\J}\bullet\left(\A^{\frac{r}{2}}_{\J}\bullet\B_{\J}^{p}\bullet\A^{\frac{r}{2}}_{\J}\right)^{\frac{r}{p+r}}_{\J}\bullet\A^{-\frac{r}{2}}_{\J}\A^{\frac{r}{2}}_{\J}\preceq_{\J} \A^{\frac{r}{2}}_{\J}\bullet\A^{\frac{r}{2}}_{\J}\,,
\end{equation*}
which implies 
\begin{equation*}
\left(\A^{\frac{r}{2}}_{\J}\bullet\B_{\J}^{p}\bullet\A^{\frac{r}{2}}_{\J}\right)^{\frac{r}{p+r}}_{\J}=\J\bullet\left(\A^{\frac{r}{2}}_{\J}\bullet\B_{\J}^{p}\bullet\A^{\frac{r}{2}}_{\J}\right)^{\frac{r}{p+r}}_{\J}\bullet\J\preceq_{\J} \A^{r}_{\J}\,,
\end{equation*}
as required.

\end{proof}

\appendix

\section{The quaternionic case}

\label{AppendixA}

\subsection{The division algebra $\mathbb{H}$}

\label{AppendixA1}

We denote by $\mathbb{H}$  the real associative algebra
\begin{equation*}
\mathbb{H} := \left\{a + b\mathbf{i} + c\mathbf{j} + d\mathbf{k}\,, a\,, b\,, c\,, d \in \mathbb{R}\right\}
\end{equation*}
generated by $1\,, \mathbf{i}\,, \mathbf{j}\,, \mathbf{k}$ subject to the relations
\begin{equation*}
\mathbf{i}^2 = \mathbf{j}^2 = \mathbf{k}^2 = \mathbf{i}\mathbf{j}\mathbf{k} = -1\,.
\end{equation*}
The algebra $\mathbb{H}$ is not commutative (for instance $\mathbf{i}\mathbf{j} = -\mathbf{j}\mathbf{i}$) and its center $\Z(\mathbb{H})$ is the set of real numbers $\mathbb{R}$\,.

\begin{remark}

Using $\mathbf{k} = \mathbf{i}\mathbf{j}$, every $q \in \mathbb{H}$ can be written uniquely as
\begin{equation*}
q = \left(a + b\mathbf{i}\right) + \left(c + d\mathbf{i}\right)\mathbf{j}\,,
\end{equation*}
i.e. 
\begin{equation*}
q = z_{1} + z_{2}\mathbf{j}\,, \qquad \left(z_{1}\,, z_{2} \in \mathbb{C}\right)\,,
\end{equation*}
where $\mathbb{C} = \mathbb{R} \oplus \mathbb{R}\mathbf{i}$ is identified with the subalgebra of $\mathbb{H}$ generated by $1$ and $\mathbf{i}$\,. 

\end{remark}

\begin{notation}

For a quaternionic number $q = a + b\mathbf{i} + c\mathbf{j} + d\mathbf{k} \in \mathbb{H}$, we denote by $\Re(q)$ and $\Im(q)$ the real and imaginary parts of $q$ respectively given by
\begin{equation*}
\Re(q) = a\,, \qquad \qquad \Im(q) = b\mathbf{i} + c\mathbf{j} + d\mathbf{k}\,.
\end{equation*}
Moreover, we denote by $\Im(\mathbb{H})$ the subset of $\mathbb{H}$ given by
\begin{equation*}
\Im(\mathbb{H}) := \left\{b\mathbf{i} + c\mathbf{j} + d\mathbf{k}\,, b\,, c\,, d \in \mathbb{R}\right\}\,.
\end{equation*}

\end{notation}

\noindent The quaternionic conjugation $\overline{q}$ of an element $q \in \mathbb{H}$ is defined by
\begin{equation*}
\overline{q} = a - b\mathbf{i} - c\mathbf{j} - d\mathbf{k}\,,
\end{equation*}
and the corresponding map $\iota$
\begin{equation}
\iota: \mathbb{H} \ni q \to \overline{q} \in \mathbb{H}
\label{MapIota}
\end{equation}
is a $\mathbb{R}$-linear involution on $\mathbb{H}$. Equivalently, if $q = z_{1} + z_{2}\mathbf{j}$, the conjugate of $q$ is given by
\begin{equation*}
\overline{q} = \overline{z_{1}} - \overline{z_{2}} \mathbf{j}\,.
\end{equation*}
The corresponding norm is given by
\begin{equation*}
\left|q\right| = \sqrt{q\overline{q}} = \sqrt{a^2 + b^2 + c^2 + d^2}\,,
\end{equation*}
and satisfies $\left|pq\right| = \left|p\right|\left|q\right|$ for all $p\,, q \in \mathbb{H}$\,.

\begin{remark}

There exists another way to write quaternionic numbers in a nice and useful form. For all $q \in \mathbb{H}$, there exist a triple $\left(a\,, b\,, u\right) \in \mathbb{R} \times \mathbb{R}_{\geq 0} \times \Im(\mathbb{H})$ such that $q = a + bu$, with $u^{2} = -1$, unique if $b > 0$.

\noindent Indeed, let $q = \Re(q) + \Im(q) \in \mathbb{H}$. If $\Im(q) = 0$, the claim is trivial ($b = 0$). Suppose that $\Im(q) \neq 0$. Let $a = \Re(q), b = \left|\Im(q)\right|,$ and $u = \frac{\Im(q)}{\left|\Im(q)\right|}$. In particular, we have $a \in \mathbb{R}, b \geq 0$, and 
\begin{eqnarray*}
q = \Re(q) + \Im(q) = \Re(q) + \left|\Im(q)\right|\frac{\Im(q)}{\left|\Im(q)\right|} = a + bu\,.
\end{eqnarray*}
Using that $\left|u\right| = 1$ and $\bar{u} = -u$ (because $u$ is purely imaginary), we get
\begin{equation*}
u^{2} = -u\overline{u} = -\left|u\right|^{2} = -1\,.
\end{equation*}

\end{remark}

\noindent Finally, the identification $\mathbb{H} \simeq \mathbb{C}^{2}$ via the map $q = z_{1} + z_{2}\mathbf{j}$ induces an injective embedding of $\mathbb{R}$-algebras
\begin{equation}
\Psi: \mathbb{H} \hookrightarrow \Mat(2\,, \mathbb{C})\,, \qquad \qquad z_{1} + z_{2}\mathbf{j} \mapsto \begin{bmatrix} z_{1} & z_{2} \\ -\overline{z}_{2} & \overline{z}_{1} \end{bmatrix}\,,
\label{EmbeddingOne}
\end{equation}
which identifies quaternionic conjugation with the Hermitian adjoint and the quaternionic norm with the determinant\,.

\begin{remark}

Using the embedding $\Psi$ defined in Equation \eqref{EmbeddingOne}, we define the reduced trace on $\mathbb{H}$ as the map
\begin{equation}
\trd: \mathbb{H} \mapsto \mathbb{R}\,, \qquad \qquad \trd(q) := \frac{1}{2}\tr\left(\Psi(q)\right)\,,
\label{ReducedTrace}
\end{equation}
where $\tr$ denotes the usual matrix trace in $\Mat(2\,, \mathbb{C})$. More precisely, if $q = z_{1} + z_{2}\mathbf{j}$ with $z_{1}\,, z_{2} \in \mathbb{C}$, then
\begin{equation*}
\trd(q) = \frac{1}{2}\left(z_{1} + \overline{z_{1}}\right) = \Re(z_{1}) \in \mathbb{R}\,.
\end{equation*}
In particular, the reduced trace coincides with the real part of a quaternion, i.e.
\begin{equation*}
\trd(q) = \frac{1}{2}\left(q + \overline{q}\right)\,.
\end{equation*}

\end{remark}

\begin{lemma}

The reduced trace $\trd$ is invariant under quaternionic conjugation (i.e. $\trd(\overline{q}) = \trd(q)$), and for all $p\,, q \in \mathbb{H}$, we have
\begin{equation*}
\trd(pq)=\trd(qp)\,.
\end{equation*}

\label{LemmaTrdOnH}

\end{lemma}

\begin{proof}

For $q \in \mathbb{H}$, we have
\begin{equation*}
\trd(\overline{q}) = \frac{1}{2}\left(\overline{q} + \overline{\overline{q}}\right) = \frac{1}{2}\left(\overline{q} + q\right) = \trd(q)\,.
\end{equation*}

\noindent Let $p\,, q \in \mathbb{H}$. We first show that $\Re(pq) = \Re(qp)$. Indeed,
\begin{equation*}
\Re(pq) = \frac{1}{2}\left(pq + \overline{pq}\right) = \frac{1}{2}\left(pq + \bar{q}\bar{p}\right)\,,
\end{equation*}
and
\begin{equation*}
\Re(qp) = \frac{1}{2}\left(qp + \overline{qp}\right) = \frac{1}{2}\left(qp + \bar{p}\bar{q}\right)\,.
\end{equation*}
We now write $p = a + u$ and $q = b + v$, with $a\,, b \in \mathbb{R}$ $u\,, v \in \Im(\mathbb{H})$. Then
\begin{equation*}
pq = ab + av + bu + uv\,, \qquad \bar{q} \bar{p} = (b-v)(a-u) = ab - au - bv + vu,
\end{equation*}
i.e.
\begin{equation*}
pq + \bar{q}\bar{p} = 2ab + (uv + vu)\,.
\end{equation*}
Similarly, we get
\begin{equation*}
qp + \bar{p}\bar{q} = 2ab + (uv + vu)\,.
\end{equation*}
Hence $pq + \bar{q}\bar{p} = qp + \bar{p}\bar{q}$, i.e. $\Re(pq) = \Re(qp)$, so $\trd(pq) = \trd(qp)$\,.

\end{proof}

\begin{corollary}

For all $p\,, q \in \mathbb{H}$, we have $pq - qp \in \Im(\mathbb{H})$\,.

\end{corollary}

\subsection{The cone of positive quaternionic hermitian matrices}

\label{SectionConeHermitianMatrices}

We keep the notations of Section \ref{SectionOne}. Let $\V$ be a right-$\mathbb{H}$-module with $\dim_{\mathbb{H}}(\V) = n$, and let $\langle\cdot\,, \cdot\rangle: \V \times \V \to \mathbb{H}$ be a positive $\left(\iota\,, 1\right)$-Hermitian form on $\V$, with $\iota(\X) = \overline{\X}$. As explained in Theorem \ref{TheoremClassificationHermitian}, there exists a basis $\mathscr{B}$ such that $\Mat_{\mathscr{B}}(\langle\cdot\,, \cdot\rangle) = \Id_{n}$\,.

\noindent We denote by $\Sp(\V)$ the subgroup of $\GL(\V)$ given by
\begin{equation*}
\Sp(\V) := \left\{\T \in \GL(\V)\,, \langle \T(u)\,, \T(v)\rangle = \langle u\,, v\rangle\,, \left(u\,, v \in \V\right)\right\}\,.
\end{equation*}
Similarly, using the basis $\mathscr{B}$ and replacing $\T$ by $\Mat_{\mathscr{B}}(\T)$, we denote by $\Sp(n\,, \mathbb{H})$ the subgroup of $\GL(n)$ given by
\begin{equation*}
\Sp(n\,, \mathbb{H}) = \left\{g \in \GL(n)\,, \iota(g)^{t}g = \Id_{n}\right\} = \left\{g \in \GL(n)\,, g^{*}g = \Id_{n}\right\}\,.
\end{equation*}
We denote by $\mathfrak{sp}(\V)$ the Lie algebra of $\Sp(\V)$, so
\begin{equation*}
\mathfrak{sp}(\V) := \left\{\X \in \End(\V)\,, \langle \X(u)\,, v\rangle + \langle u\,, \X(v)\rangle = 0\,, \left(u\,, v \in \V\right)\right\}\,,
\end{equation*}
and let $\mathfrak{sp}(n\,, \mathbb{H}) = \left\{\X \in \Mat(n)\,, \X^{*} + \X = 0\right\}$\,.

\noindent To simplify the notations, we will write $\Sp(n)$ and $\mathfrak{sp}(n)$ instead of $\Sp(n\,, \mathbb{H})$ and $\mathfrak{sp}(n\,, \mathbb{H})$\,.

\begin{remark}

\begin{enumerate}
\item The quaternionic symplectic group $\Sp(n)$ is connected and compact (see \cite{KNAPP})\,.
\item For two matrices $\A\,, \B \in \Mat(n)$, we have $\left(\A\B\right)^{t} \neq \B^{t}\A^{t}$ in general. However, we have $\left(\A\B\right)^{*} = \B^{*}\A^{*}$. Moreover, if $\A\,, \B$ are invertible, we have $\iota(\A)^{-1} \neq \iota(\A^{-1})$ in general, but
\begin{equation*}
\left(\A^{*}\right)^{-1} = \left(\A^{-1}\right)^{*}\,, \qquad \qquad \left(\A\B\right)^{-1} = \B^{-1}\A^{-1}\,.
\end{equation*}
The proof of this result can be found in \cite[Theorem~4.1]{ZHANG}\,.
\end{enumerate}

\label{RemarkPropertiesMatrices}

\end{remark}

\noindent We denote by $\mathfrak{p}$ the subset of $\Mat(n)$ given by
\begin{equation*}
\mathfrak{p} = \left\{\X \in \Mat(n)\,, \X = \X^{*}\right\}\,.
\end{equation*}
The set $\mathfrak{p}$ is known as the set of quaternionic hermitian matrices. Moreover, we get
\begin{equation*}
\left[\mathfrak{sp}(n)\,, \mathfrak{sp}(n)\right] \subseteq \mathfrak{sp}(n)\,, \qquad \left[\mathfrak{sp}(n)\,, \mathfrak{p}\right] \subseteq \mathfrak{p}\,, \qquad \left[\mathfrak{p}\,, \mathfrak{p}\right] \subseteq \mathfrak{sp}(n)\,,
\end{equation*}
where $\left[\cdot\,, \cdot\right]$ is the Lie bracket on $\Mat(n)$ given by $\left[\A\,, \B\right] = \A\B - \B\A$\,.

\medskip

\noindent We now recall some results of \cite{ZHANG} concerning eigenvalues of quaternionic matrices. In this paper, we will only deal with right eigenvalues, i.e. that we say that $\lambda \in \mathbb{H}$ is an eigenvalue of $\A \in \Mat(n)$ is there exists a non-zero vector $v \in \mathbb{H}^{n}$ such that $\A v = v\lambda$. 

\noindent The main difference between quaternionic and real/complex matrices is that in general, a matrix has infinitely many eigenvalues. Indeed, let $\A \in \Mat(n)$ and $\lambda \in \mathbb{H}$ be a right eigenvalue of $\A$. Then for every $q \in \mathbb{H}^{*}$, the quaternion $q^{-1} \lambda q$ is also a right eigenvalue of $\A$. Indeed, let $v \in \mathbb{H}^{n} \setminus\{0\}$ satisfying $\A v = v \lambda$. For any $q \in \mathbb{H}^{*}$, define $w := v q$. Then $w \neq 0$ and
\begin{equation*}
\A w = \A(v q) = (\A v) q = v \lambda q = (v q)(q^{-1}\lambda q) = w (q^{-1}\lambda q)\,,
\end{equation*}
which shows that $q^{-1}\lambda q$ is also a right eigenvalue of $\A$\,. 

\begin{remark}

For all $\lambda \in \mathbb{H}$, we denote by $\left[\lambda\right]$ the subset of $\mathbb{H}^{n}$ given by
\begin{equation*}
\left[\lambda\right] := \left\{q\lambda q^{-1}\,, q \in \mathbb{H}^{*}\right\}\,.
\end{equation*}
As explained in \cite{ZHANG}, for all $q \in \mathbb{H}$ such that $\Im(q) \neq 0$, then the set
\begin{equation*}
\left[q\right] \cap \left\{z \in \mathbb{C}\,, \Im(z) > 0\right\}
\end{equation*}
has a unique element. Moreover, every $\A \in \Mat(n)$ has exactly $n$ right-eigenvalues (up to conjugation) and the previous equation give us a canonical choice for the spectrum of $\A$, made of complex numbers with positive imaginary parts.

\label{RemarkNEigenvalues}

\end{remark}

\noindent However, the situation is slightly easier for Hermitian matrices\,.

\begin{lemma}

Let $\A \in \mathfrak{p} \subseteq \Mat(n)$. Then every right eigenvalue of $\A$ is real\,.

\label{LemmaRealSpectrum}

\end{lemma}

\begin{proof}

We use the positive hermitian form on $\mathbb{H}^{n}$ defines above
\begin{equation*}
\langle x\,, y\rangle = x^{*}y\,.
\end{equation*}
Let $v \in \mathbb{H}^{n}$ non-zero and $\lambda \in \mathbb{H}$ satisfying $\A v = v\lambda$. Using that $\A = \A^{*}$, we get
\begin{equation*}
\langle \A x\,, y\rangle = \langle x\,, \A y\rangle\,, \qquad \left(x\,, y \in \mathbb{H}^{n}\right)\,.
\end{equation*}
In particular,
\begin{equation*}
\langle \A v\,, v\rangle = \langle v\,,\A v\rangle.
\end{equation*}
Using that $\A v = v\lambda$, we get that
\begin{equation*}
\langle \A v\,, v\rangle = \langle v\lambda\,, v\rangle = \overline{\lambda}\langle v\,, v\rangle\,,
\end{equation*}
and
\begin{equation*}
\langle v,\A v\rangle = \langle v, v\lambda\rangle = \langle v\,, v\rangle\lambda\,.
\end{equation*}
Therefore
\begin{equation*}
\overline{\lambda}\langle v\,,v\rangle = \langle v\,,v\rangle\lambda\,.
\end{equation*}
Finally, using that $\langle v\,, v\rangle \in \mathbb{R}^{*}$ and that $\mathbb{R}$ commute with $\mathbb{H}$, we get $\overline{\lambda} = \lambda$. Thus $\lambda\in\mathbb{R}$\,.

\end{proof}

\begin{definition}

We say that a matrix $\X \in \Mat(n)$ is positive if $\langle\X v\,, v\rangle > 0$ for all non-zero $v \in \mathbb{H}^{n}$\,.

\end{definition}

\noindent We denote by $\mathscr{P}$ the cone of positive Hermitian matrices, and let 
\begin{equation*}
\exp: \Mat(n) \mapsto \GL(n)
\end{equation*}
be the exponential map\,.

\begin{lemma}

For all $\X \in \mathfrak{p}$, we have $\exp(\X) \in \mathscr{P}$\,.

\end{lemma}

\begin{proof}

Let $\X \in \mathfrak{p}$, i.e. $\X = \X^{*}$. Using that for all $k \geq 0$, we have $\left(\X^{*}\right)^{k} = \left(\X^{k}\right)^{*}$, we get that $\exp(\X)^{*} = \exp(\X^{*}) = \exp(\X)$, i.e. $\exp(\X) \in \mathfrak{p}$. Moreover, it follows from Remark \ref{RemarkNEigenvalues} and Lemma \ref{LemmaRealSpectrum} that $\Spec(\X) = \left\{\lambda_{1}\,, \lambda_{2}\,, \ldots\,, \lambda_{n}\right\}$, with $\lambda_{i} \in \mathbb{R}\,, 1 \leq i \leq n$. Therefore, $\Spec(\exp(\X)) = \left\{e^{\lambda_{1}}\,, \ldots\,, e^{\lambda_{n}}\right\} \subseteq \mathbb{R}_{>0}$, i.e. $\exp(\X)$ is positive. Then $\exp(\X) \in \mathscr{P}$\,.

\end{proof}

\noindent In the following theorem, we recall an important result of \cite[Corollary~6.2]{ZHANG} that is crucial to show that $\exp: \mathfrak{p} \to \mathscr{P}$ is bijective.

\begin{theo}

Let $\X \in \mathfrak{p}$. Then there exist a matrix $\U \in \Sp(n)$ and real numbers $\lambda_{1}\,, \ldots\,, \lambda_{n} \in \mathbb{R}$ such that
\begin{equation*}
\X = \U\diag(\lambda_{1}\,, \ldots\,, \lambda_{n})\U^{*}\,.
\end{equation*}

\label{SpectralTheorem}

\end{theo}

\noindent We can now prove the following result\,.

\begin{theo}

The matrix exponential
\begin{equation*}
\exp:\mathfrak{p} \mapsto \mathscr{P}\,, \qquad \qquad \H \mapsto \exp(\H)
\end{equation*}
is well-defined and bijective\,. 

\label{ExponentialTheorem}

\end{theo}

\begin{proof}

We first prove the surjectivity of $\exp: \mathfrak{p} \to \mathscr{P}$. Let $\P \in \mathscr{P}$. By the quaternionic spectral theorem, there exist $\U \in \Sp(n)$ and real numbers $\mu_{1}\,, \ldots\,, \mu_{n} > 0$ such that
\begin{equation*}
\P = \U\diag(\mu_{1}\,, \ldots\,, \mu_{n})U^{*}\,.
\end{equation*}
Define
\begin{equation*}
\H := \U\diag(\log(\mu_{1})\,, \ldots\,, \log(\mu_{n}))\U^{*}\,.
\end{equation*}
Then $\H^{*} = \H$, i.e. $\H \in \mathfrak{p}$, and
\begin{equation*}
\exp(\H) = \U\diag(e^{\log(\mu_{1})}\,, \ldots\,, e^{\log(\mu_{n})})\U^{*} = \U\diag(\mu_{1}\,, \ldots\,, \mu_{n})\U^{*} = \P\,.
\end{equation*}
Thus $\exp$ is surjective\,.

\noindent We now prove that injectivity of $\exp$ on $\mathfrak{p}$. Assume $\H_{1}\,, \H_{2} \in \mathfrak{p}$ are such that $\exp(\H_{1}) = \exp(\H_{2})$. By the spectral theorem, we take $\U_{1}\,, \U_{2} \in \Sp(n\,, \mathbb{H})$ and real numbers $\lambda_{1}\,, \ldots\,, \lambda_{n}\,, \nu_{1}\,, \ldots\,, \nu_{n}$ such that
\begin{equation*}
\H_{1} = \U_{1} \diag(\lambda_{1}\,, \ldots\,, \lambda_{n}) \U^{*}_{1}\,, \qquad
\H_{2} = \U_{2} \diag(\nu_{1}\,, \ldots\,, \nu_{n}) \U^{*}_{2}\,.
\end{equation*}
Then
\begin{equation*}
\exp(\H_{1}) = \U_{1} \diag(e^{\lambda_{1}}\,, \ldots\,, e^{\lambda_{n}})\U^{*}_{1}\,, \qquad \exp(\H_{2}) = \U_{2} \diag(e^{\nu_{1}}\,, \ldots\,, e^{\nu_{n}})\U^{*}_{2}\,.
\end{equation*}
Since $\exp(\H_{1}) = \exp(\H_{2})$ is a positive Hermitian matrix, its spectral decomposition is unique up to permutation of the diagonal entries and unitary change of basis within each eigenspace. In particular, we have an equality of sets
\begin{equation*}
\left\{e^{\lambda_{1}}\,, \ldots\,, e^{\lambda_{n}}\right\} = \left\{e^{\nu_{1}}\,, \ldots\,, e^{\nu_{n}}\right\}\,.
\end{equation*}
Because the real exponential $t \mapsto e^{t}$ is injective on $\mathbb{R}$, it follows that
\begin{equation*}
\left\{\lambda_{1}\,, \ldots\,, \lambda_{n}\right\} = \left\{\nu_{1}\,, \ldots\,, \nu_{n}\right\}\,.
\end{equation*}
Hence $\H_{1}$ and $\H_{2}$ are unitarily diagonalizable with the same real eigenvalues, so $\H_{1} = \H_{2}$\,.

\end{proof}

\noindent In particular, it follows from Theorem \ref{ExponentialTheorem} that the map $\exp: \mathfrak{p} \to \mathscr{P}$ is invertible. Let $\log := \exp^{-1}: \mathscr{P} \to \mathfrak{p}$ be the inverse of $\exp$. 

\begin{definition}

For all $t \in \mathbb{R}$ and $\X \in \mathscr{P}$, we denote by $\X^{t}$ the matrix in $\mathscr{P}$ given by
\begin{equation*}
\X^{t} := \exp\left(t\log(\X)\right)\,.
\end{equation*}
If $t = \frac{1}{2}$, the matrix $\X^{\frac{1}{2}}$ is the square root of $\X$\,.

\end{definition}

\noindent We finish this section with a lemma\,.

\begin{lemma}

For all $g \in \GL(n)$ and $\P \in \mathscr{P}$, we have $g\P g^{*} \in \mathscr{P}$. Moreover, the corresponding action of $\GL(n)$ on $\mathscr{P}$ is transitive\,.

\label{LemmaTransitiveAction}

\end{lemma}

\begin{proof}

Using Remark \ref{RemarkPropertiesMatrices}, it follows that for all $g \in \GL(n)$ and $\P \in \mathscr{P}$ (i.e. $\P = \P^{*}$ and $\P > 0$), we get
\begin{equation*}
\left(g\P g^{*}\right)^* = \left(g^{*}\right)^{*}\P^{*}g^{*} = g\P g^{*}\,,
\end{equation*}
i.e. $g\P g^{*} \in \mathfrak{p}$. Moreover, using that $\P > 0$, we get that for all non-zero $x \in \mathbb{H}^{n}$
\begin{equation*}
x^{*}\left(g\P g^{*}\right)x = \left(x^{*}g\right)\P\left(g^{*}x\right) = \left(g^{*}x\right)^{*}\P\left(g^{*}x\right) > 0\,,
\end{equation*}
i.e. $g\P g^{*} \in \mathscr{P}$\,.

\noindent To prove that the corresponding action $\GL(n) \curvearrowright \mathscr{P}$ is transitive, it is enough to prove that every $\P \in \mathscr{P}$ is in the orbit $\mathscr{O}_{\Id_{n}}$ of the identity matrix $\Id_{n} \in \mathscr{P}$. Let $\P \in \mathscr{P}$. Using Theorem \ref{SpectralTheorem}, there exists a matrix $\U \in \Sp(n)$ and a diagonal matrix $\D = \diag\left(\lambda_{1}\,, \dots\,, \lambda_{n}\right) \in \Mat(n)$ with $\lambda_{i} > 0$ such that
$\P = \U\D\U^{*}$. Define $\D^{\frac{1}{2}} := \diag\left(\sqrt{\lambda_{1}}\,, \dots\,, \sqrt{\lambda_{n}}\right)$ and let $g := \U\D^{\frac{1}{2}} \in \GL(n)$. Then
\begin{equation*}
gg^{*} = \U\D^{\frac{1}{2}}\left(\U\D^{\frac{1}{2}}\right)^{*} = \U\D^{\frac{1}{2}} \left(\D^{\frac{1}{2}}\right)^{*}\U^{*} = \U\D^{\frac{1}{2}}\D^{\frac{1}{2}}\U^{*} = \U\D\U^{*} = \P\,,
\end{equation*}
i.e. $\P \in \mathscr{O}_{\Id_{n}}$, and the lemma follows\,.

\end{proof}

\begin{remark}

The cone $\mathscr{P}$ is a Riemannian symmetric space of non-compact type, i.e.
\begin{equation*}
\mathscr{P} \cong \GL(n) / \Sp(n)\,.
\end{equation*}
A detailed description of such spaces can be found in \cite{HELGASON}\,.

\end{remark}

\subsection{A Riemannian form on $\mathscr{P}$}

For a matrix $\A \in \Mat(n)$, we denote by $\tr_{\mathbb{H}}$ the standard trace of $\A$, i.e.
\begin{equation*}
\tr_{\mathbb{H}}(\A) = \sum\limits_{i = 1}^{n}a_{i\,, i}\,.
\end{equation*}

\begin{remark}

Using that $\mathbb{H}$ is not commutative, it follows that in general
\begin{equation*}
\tr_{\mathbb{H}}(\A\B) \neq \tr_{\mathbb{H}}(\B\A)\,,
\end{equation*}
for $\A,\B \in \Mat(n)$\,.

\end{remark}

\noindent One way to fix it is to use the reduced trace instead of the standard trace. In Equation \eqref{ReducedTrace}, we define the reduced trace $\trd$ on $\mathbb{H}$. The reduced trace can be extended to $\Mat(n)$ by
\begin{equation*}
\trd(\A) := \trd(\tr_{\mathbb{H}}(\A)) = \sum\limits_{i = 1}^{n} \trd(a_{i, i})\,, \qquad \left(\A \in \Mat(n)\right)\,.
\end{equation*}

\begin{lemma}

For all $\A\,, \B \in \Mat(n)$, we have
\begin{equation*}
\trd(\A\B) = \trd(\B\A)\,.
\end{equation*}

\label{LemmaTrdMatrices}

\end{lemma}

\begin{proof}

Write $\A = \left(a_{i,j}\right)_{i,j}$ and $\B = \left(b_{i,j}\right)_{i,j}$. By definition of matrix multiplication, the diagonal entries of $\A\B$ are
\begin{equation*}
(\A\B)_{i,i} = \sum\limits_{j = 1}^{n} a_{i,j}b_{j,i} \qquad \left(1 \leq i \leq n\right)\,.
\end{equation*}
Using additivity of the reduced trace $\trd$, we get
\begin{eqnarray*}
\trd(\A\B) & = & \sum\limits_{i=1}^{n} \trd\big((\A\B)_{i,i}\big) = \sum\limits_{i=1}^{n} \trd\left(\sum\limits_{j=1}^{n} a_{i,j}b_{j,i}\right) \\
& = & \sum\limits_{i=1}^{n} \sum\limits_{j=1}^{n} \trd(a_{i,j}b_{j,i})\,.
\end{eqnarray*}
Now, using Lemma \ref{LemmaTrdOnH}, we can swap the factors inside $\trd$:
\begin{equation*}
\trd(a_{i,j}b_{j,i})=\trd(b_{j,i}a_{i,j})\,,
\end{equation*}
so
\begin{equation*}
\trd(\A\B) = \sum\limits_{i=1}^{n} \sum\limits_{j=1}^{n} \trd(b_{j,i}a_{i,j})\,.
\end{equation*}
Similarly, we get
\begin{equation*}
(\B\A)_{j,j} = \sum\limits_{i = 1}^{n} b_{j,i}a_{i,j} \qquad \left(1 \leq j \leq n\right)\,,
\end{equation*}
so that
\begin{equation*}
\trd(\B\A) =\sum\limits_{j=1}^{n} \trd\big((\B\A)_{j,j}\big) = \sum\limits_{j=1}^{n} \sum\limits_{i=1}^{n} \trd(b_{j,i}a_{i,j})\,.
\end{equation*}
Finally, $\trd(\A\B)=\trd(\B\A)$\,.

\end{proof}

\noindent For every $\A \in \mathscr{P}$, we identify the tangent space as
\begin{equation*}
\T_{\A}(\mathscr{P}) \cong \mathfrak{p}\,.
\end{equation*}
At the identity $\Id_{n} \in \mathscr{P}$, we define a bilinear form $\omega := \omega_{\Id_{n}}$ on $\T_{\Id_{n}}(\mathscr{P}) \cong \mathfrak{p}$ by
\begin{equation}
\label{MetricAtI}
\omega(\X\,, \Y):= \trd(\X\Y)\,, \qquad \qquad \left(\X\,, \Y \in \mathfrak{p}\right)\,.
\end{equation}
The form is real-valued, symmetric (see Lemma \ref{LemmaTrdMatrices}), and positive. Indeed, for all non-zero $\X \in \mathfrak{p}$, it follows from Theorem \ref{SpectralTheorem} that
\begin{equation*}
\omega(\X\,, \X) = \trd(\X^{2}) = \sum\limits_{j=1}^{n} \lambda^{2}_{j} > 0\,,
\end{equation*}
where $\Spec(\X) = \left\{\lambda_{1}\,, \ldots\,, \lambda_{n}\right\} \subseteq \mathbb{R}$ are the eigenvalues of $\X$. Moreover, for all $g \in \Sp(n)$ and $\X\,, \Y \in \mathfrak{p}$, we have
\begin{equation*}
\omega(g\X g^{-1}\,, g\Y g^{-1}) = \trd\left((g\X g^{-1})(g\Y g^{-1})\right) = \trd(g\X\Y g^{-1}) = \trd(\X\Y) = \omega(\X\,, \Y)\,,
\end{equation*}
i.e. the form $\omega := \omega_{\Id_{n}}$ is $\Sp(n)$-invariant\,.

\noindent We now use the transitive action of $\GL(n)$ on $\mathscr{P}$ (see Lemma \ref{LemmaTransitiveAction}) to define a Riemannian metric on $\mathscr{P}$. For all $g \in \GL(n)$, we denote by $\Phi_{g}: \mathscr{P} \to \mathscr{P}$ the map given by
\begin{equation*}
\Phi_{g}(\X) = g\X g^{*}\,, \qquad \left(\X \in \mathscr{P}\right)\,.
\end{equation*}

\noindent Let $g \in \GL(n)$ and $\A \in \mathscr{P}$. For $\X \in \T_{\A}(\mathscr{P}) \cong \mathfrak{p}$, we consider the curve 
\begin{equation*}
\gamma^{\X}_{\A}(t) = \A + t\X\,.
\end{equation*}
Then
\begin{equation*}
\left(d\Phi_{g}\right)_{\A}(\X) =\left.\frac{d}{dt}\right|_{t=0} \Phi_{g}(\gamma^{\X}_{\A}(t)) =\left.\frac{d}{dt}\right|_{t=0} g\left(\A+t\X\right)g^{*} = g\X g^{*}\,.
\end{equation*}
Then
\begin{equation}
\label{DerivativePhiG}
\left(d\Phi_{g}\right)_{\A}: \T_{\A}(\mathscr{P}) \mapsto \T_{\Phi_{g}(\A)}(\mathscr{P})\,, \qquad \X \mapsto g\X g^{*}\,.
\end{equation}
Since $\A = \Phi_{\A^{\frac{1}{2}}}(\Id_{n})$, we use $\A^{-\frac{1}{2}}$ to pull tangent vectors at $\A$ back to the identity and we get
\begin{equation*}
\omega(\X\,, \Y)_{\A} := \omega\left((d\Phi_{\A^{-\frac{1}{2}}})_{\A}(\X)\,, (d\Phi_{\A^{-\frac{1}{2}}})_{\A}(\Y)\right)_{\Id_{n}}\,.
\end{equation*}
Using the definition of the metric at the identity (see Equation \eqref{MetricAtI}), this gives
\begin{equation*}
\omega(\X\,, \Y)_{\A} = \trd\Big((\A^{-\frac{1}{2}}\X\A^{-\frac{1}{2}})(\A^{-\frac{1}{2}}\Y\A^{-\frac{1}{2}})\Big)\,.
\end{equation*}
Regrouping the factors under the trace yields to
\begin{equation}
\omega(\X\,, \Y)_{\A} = \trd\big(\A^{-1}\X\A^{-1}\Y\big)\,.
\label{FormOmegaP}
\end{equation}

\begin{lemma}

The family of positive symmetric bilinear forms $\left(\omega_{\P}\right)_{\P \in \mathscr{P}}$ defines a Riemannian metric on $\mathscr{P}$\,.

\end{lemma}

\subsection{Geodesics on $\mathscr{P}$}

\label{SectionGeodesicsOnP}

The goal of this section is to prove the following theorem\,.

\begin{theo}

Let $\A\,, \B \in \mathscr{P}$. Then the unique geodesic (with respect to the Riemannian metric $\omega$) joining $\A$ and $\B$ in $\mathscr{P}$ is given by
\begin{equation}
\gamma(t) = \A^{\frac{1}{2}}\left(\A^{-\frac{1}{2}}\B\A^{-\frac{1}{2}}\right)^{t}\A^{\frac{1}{2}}\,.
\label{GeodesicInPH}
\end{equation}

\label{MainTheorem}

\end{theo}

\noindent Before proving Theorem \ref{MainTheorem}, we will have to prove a few intermediate results. Let $\Psi: \mathbb{H} \to \Mat(2\,, \mathbb{C})$ be the map given in Equation \eqref{EmbeddingOne}. Using that $\mathbb{H} = \mathbb{C} \oplus \mathbb{C}\mathbf{j}$, we obtain that every matrix $\X \in \Mat(n)$ can be written as $\X = \A + \B\mathbf{j}$, with $\A\,, \B \in \Mat(n\,, \mathbb{C})$. The map $\Psi$ can be extended to a map $\Psi: \Mat(n) \mapsto \Mat(2n\,, \mathbb{C})$ as
\begin{equation}
\Psi(\A + \B\mathbf{j}) = \begin{bmatrix} \A & \B \\ -\overline{\B} & \overline{\A} \end{bmatrix}\,, \qquad \qquad \left(\A\,, \B \in \Mat(n\,, \mathbb{C})\right)\,.
\label{EquationPsiMatrices}
\end{equation}
For all $\X\,, \Y \in \Mat(n)$, we have $\Psi(\X\Y) = \Psi(\X)\Psi(\Y)$. Moreover, if $\X$ is invertible, then $\Psi(\X)$ is invertible and we have $\Psi(\X^{-1}) = \Psi(\X)^{-1}$\,.

\noindent For a matrix $\X \in \Mat(2n\,, \mathbb{C})$, we denote by $\X^{*}$ the matrix $\X^{*} = \overline{\X}^{t}$. In particular, we have $\Psi(\X^{*}) = \Psi(\X)^{*}$.

\begin{notation}

We denote by $\J_{2n}$ the matrix of $\Mat(2n\,, \mathbb{C})$ given by
\begin{equation*}
\J_{2n} := \begin{bmatrix} 0 & \Id_{n} \\ -\Id_{n} & 0 \end{bmatrix}\,.
\end{equation*}
In particular, $\J^{2}_{2n} = -\Id_{2n}$\,.

\end{notation}

\begin{lemma}

We get $\Psi(\Mat(n)) = \left\{\M \in \Mat(2n\,, \mathbb{C})\,, \M\J_{2n} = \J_{2n}\overline{\M}\right\}$\,.

\end{lemma}

\noindent In the following lemmas, we will give an explicit description of $\Psi(\Sp(n))$ and $\Psi(\mathscr{P})$\,.

\begin{remark}

Let $\langle\cdot\,, \cdot\rangle$ be the hermitian form on the right-$\mathbb{H}$-module $\V = \mathbb{H}^{n}$ given in Section \ref{SectionConeHermitianMatrices}. We denote by $\V_{\mathbb{C}}$ the vector space $\V$ viewed as a complex vector space. In particular, we have $\dim_{\mathbb{C}}(\V_{\mathbb{C}}) = 2n$. We denote by $\B$ the form on $\V_{\mathbb{C}}$ given by
\begin{equation*}
\B(u\,, v) = \pr_{\mathbb{C}}(\langle u\,, v\rangle)\,, \qquad \qquad \left(u\,, v \in \V_{\mathbb{C}}\right)\,,
\end{equation*}
where $\pr_{\mathbb{C}}: \mathbb{H} \to \mathbb{C}$ is the map given by $\pr_{\mathbb{C}}(a + b\mathbf{j}) = a$\,.

\noindent For all $u\,, v\,, w \in \V_{\mathbb{C}}$ and $\lambda \in \mathbb{C}$, we have
\begin{eqnarray*}
\B(u\,, v + w\lambda) & = & \pr_{\mathbb{C}}(\langle u\,, v + w\lambda\rangle) = \pr_{\mathbb{C}}(\langle u\,, v\rangle + \langle u\,, w\rangle\lambda) \\
& = & \pr_{\mathbb{C}}(\langle u\,, v\rangle) + \pr_{\mathbb{C}}(\langle u\,, w\rangle)\lambda = \B(u\,, v) + \B(u\,, w)\lambda\,,
\end{eqnarray*}
and 
\begin{eqnarray*}
\B(u + v\lambda\,, w) & = & \pr_{\mathbb{C}}(\langle u + v\lambda\,, w\rangle) = \pr_{\mathbb{C}}(\langle u\,, w\rangle + \overline{\lambda}\langle v\,, w\rangle) \\
& = & \pr_{\mathbb{C}}(\langle u\,, w\rangle) + \overline{\lambda}\pr_{\mathbb{C}}(\langle v\,, w\rangle) = \B(u\,, w) + \overline{\lambda}\B(v\,, w)\,,
\end{eqnarray*}
i.e. $\B$ is a sesquilinear form on $\V_{\mathbb{C}}$. Moreover, using that $\pr_{\mathbb{C}}(\iota(q)) = \overline{\pr_{\mathbb{C}}(q)}$ for all $q \in \mathbb{H}$, we get
\begin{equation*}
\B(u\,, v) = \pr_{\mathbb{C}}(\langle u\,, v\rangle) = \pr_{\mathbb{C}}\left(\iota(\langle v\,, u\rangle)\right) = \pr_{\mathbb{C}}\left(\overline{\langle v\,, u\rangle}\right) = \overline{\B(v\,, u)}\,,
\end{equation*}
so $\B$ is Hermitian. Finally, using that $\langle u\,, u\rangle > 0$ for all $u$-non-zero, it follows that the form $\B$ is positive. We denote by $\U(\V_{\mathbb{C}})$ the subgroup of $\GL(\V_{\mathbb{C}})$ given by
\begin{equation*}
\U(\V_{\mathbb{C}}) := \left\{g \in \GL(\V_{\mathbb{C}})\,, \B(gu\,, gv) = \B(u\,, v)\,, \left(u,, v \in \V_{\mathbb{C}}\right)\right\}\,.
\end{equation*}
Using that the form $\langle\cdot\,, \cdot\rangle$ is $\Sp(n)$-invariant, it follows that $\Psi(\Sp(n))$ is a subgroup of $\U(\V_{\mathbb{C}})$. 

\noindent Similarly, let $\J$ be the map on $\V_{\mathbb{C}}$ given by
\begin{equation*}
\J(v) = v\mathbf{j}\,, \qquad \qquad \left(v \in \V_{\mathbb{C}}\right)\,,
\end{equation*}
and let $\C: \V_{\mathbb{C}} \times \V_{\mathbb{C}} \to \mathbb{C}$ be the form given by 
\begin{equation*}
\C(u\,, v) = \B(\J(u)\,, v)\,, \qquad \qquad \left(u\,, v \in \V_{\mathbb{C}}\right)\,.
\end{equation*}
Using that $z\mathbf{j} = \mathbf{j}\overline{z}$ for all $z \in \mathbb{C}$, it follows that $\J(uz) = \J(u)\overline{z}$. Therefore for all $u\,, v \in \V_{\mathbb{C}}$ and $z \in \mathbb{C}$, we have
\begin{equation*}
\C(uz\,, v) = z\C(u\,, v)\,, \qquad \qquad \C(u, vz)  = \C(u\,, v)z\,,
\end{equation*}
i.e. $\C$ is bilinear. Moreover, we have
\begin{eqnarray*}
\C(u\,, v) & = & \B(\J(u)\,, v) = \overline{\B(v\,, \J(u))} = \overline{\pr_{\mathbb{C}}(\langle v\,, u\mathbf{j}\rangle)} = \overline{\pr_{\mathbb{C}}(\langle v\,, u\rangle\mathbf{j})} \\
& = & \overline{\pr_{\mathbb{C}}(\mathbf{j}\overline{\langle v\,, u\rangle})} = -\overline{\pr_{\mathbb{C}}(\overline{\langle v\mathbf{j}\,, u\rangle})} = -\overline{\pr_{\mathbb{C}}(\overline{\langle \J(v)\,, u\rangle})} \\
& = & -\overline{\overline{\C(v\,, u)}} = -\C(v\,, u)
\end{eqnarray*}
i.e. $\C$ is skew-symmetric. Moreover, using that $\B$ is non-degenerate, it follows that $\C$ is non-degenerate. We denote by $\Sp(\V_{\mathbb{C}})$ the subgroup of $\GL(\V_{\mathbb{C}})$ given by
\begin{equation*}
\Sp(\V_{\mathbb{C}}) := \left\{g \in \GL(\V_{\mathbb{C}})\,, \C(g(u)\,, g(v)) = \C(u\,, v)\,, \left(u\,, v \in \V_{\mathbb{C}}\right)\right\}\,,
\end{equation*}
and $\Psi(\Sp(n))$ is a subgroup of $\Sp(\V_{\mathbb{C}})$\,.

\noindent If $\mathscr{B} = \left\{v_{1}\,, \ldots\,, v_{n}\right\}$ is a right $\mathbb{H}$-basis of $\V = \mathbb{H}^{n}$, then $\mathscr{B}_{\mathbb{C}} = \left\{v_{1}\,, \ldots\,, v_{n}\,, v_{1}\mathbf{j}\,, \ldots\,, v_{n}\mathbf{j}\right\}$ is a complex basis of $\V_{\mathbb{C}}$. Moreover, one can see that
\begin{equation*}
\Mat_{\mathscr{B}_{\mathbb{C}}}(\B) = \Id_{2n}\,, \qquad \qquad \Mat_{\mathscr{B}_{\mathbb{C}}}(\C) = \J_{2n}\,.
\end{equation*}
Therefore,
\begin{equation*}
\U(\V_{\mathbb{C}}) \cong \U(2n\,, \mathbb{C}) = \left\{g \in \GL(2n\,, \mathbb{C})\,, g^{*}g = \Id_{2n}\right\}\,,
\end{equation*}
and
\begin{equation*}
\Sp(\V_{\mathbb{C}}) \cong \Sp(2n\,, \mathbb{C}) = \left\{g \in \GL(2n\,, \mathbb{C})\,, g^{t}\J_{2n}g = \J_{2n}\right\}\,.
\end{equation*}

\label{RemarkSpnH}

\end{remark}

\begin{lemma}

We have $\Psi(\Sp(n)) = \U(2n\,, \mathbb{C}) \cap \Sp(2n\,, \mathbb{C})$\,.

\end{lemma}

\begin{proof}

As explained in Remark \ref{RemarkSpnH}, we have $\Psi(\Sp(n)) \subseteq \U(2n\,, \mathbb{C}) \cap \Sp(2n\,, \mathbb{C})$. To prove the other inclusion, using that the three groups $\Sp(n)\,, \U(2n\,, \mathbb{C})\,$ and $\Sp(2n\,, \mathbb{C})$ are connected, it is enough to prove that 
\begin{equation*}
\dim_{\mathbb{R}}(\Lie(\Sp(n))) = \dim_{\mathbb{R}}(\Lie\left(\U(2n\,, \mathbb{C}) \cap \Sp(2n\,, \mathbb{C})\right))\,.
\end{equation*}
Using that 
\begin{eqnarray*}
\dim_{\mathbb{R}}(\Lie\left(\U(2n\,, \mathbb{C}) \cap \Sp(2n\,, \mathbb{C})\right)) & = & \dim_{\mathbb{C}}\left(\left(\mathfrak{u}(2n\,, \mathbb{C}) \cap \mathfrak{sp}(2n\,, \mathbb{C})\right) \otimes_{\mathbb{R}} \mathbb{C}\right) \\
& = & \dim_{\mathbb{C}}\left(\mathfrak{sp}(2n\,, \mathbb{C})\right)\,,
\end{eqnarray*}
it follows that $\dim_{\mathbb{C}}\left(\mathfrak{sp}(2n\,, \mathbb{C})\right) = \dim_{\mathbb{C}}(\S^{2}(\mathbb{C}^{2n})) = \frac{2n(2n+1)}{2} = n(2n+1)$. Similarly, one can easily see that 
\begin{equation*}
\dim_{\mathbb{R}}(\Lie(\Sp(n))) = 3n + 4\frac{(n-1)n}{2} = 3n + 2n(n-1) = n(2n-2+3) = n(2n+1)\,,
\end{equation*}
and the lemma follows\,.

\end{proof}

\noindent Similarly, denote by $\mathfrak{s}$ the set of Hermitian matrices in $\Mat(2n\,, \mathbb{C})$, i.e.
\begin{equation*}
\mathfrak{s} := \left\{\X \in \Mat(2n\,, \mathbb{C})\,, \X^{*} = \X\right\}\,,
\end{equation*}
and by $\mathscr{S}$ the set of positive matrices in $\mathfrak{s}$\,.

\begin{remark}

The complex cone $\mathscr{S}$ is exactly the one we define in Section \ref{SectionOne}; we changed the notation to avoid any confusion\,.

\end{remark}

\begin{lemma}

We get
\begin{equation}
\Psi(\mathscr{P}) = \left\{\X \in \mathscr{S}\,, \X\J_{2n} = \J_{2n}\overline{\X}\right\}\,.
\label{EquationImageCone}
\end{equation}
In particular, $\Psi(\mathscr{P})$ is a closed submanifold of $\mathscr{S}$\,.

\end{lemma}

\begin{proof}

Write $\X \in \Mat(2n\,, \mathbb{C})$ in blocks $\X = \begin{bmatrix}\P & \Q \\ \R & \S\end{bmatrix}$. A direct computation gives
\begin{equation*}
\X\J=\begin{bmatrix} -\Q & \P \\ -\S & \R \end{bmatrix}\,, \qquad \J\overline{\X} = \begin{bmatrix} \overline{\R} & \overline{\S} \\ -\overline{\P} & -\overline{\Q} \end{bmatrix}\,.
\end{equation*}
Hence $\X\J = \J\overline{\X}$ is equivalent to $\R = -\overline{\Q}$ and $\S = \overline{\P}$, i.e.
\begin{equation*}
\X = \begin{bmatrix} \P & \Q \\ -\overline{\Q} & \overline{\P} \end{bmatrix} = \Psi\left(\P + \Q\mathbf{j}\right) \in \Psi\left(\Mat(n)\right)\,.
\end{equation*}
Intersecting with $\mathscr{S}$ (i.e.\ imposing $\X = \X^{*}$ with $\X$ positive) yields $\X = \Psi(\A)$ with $\A = \A^{*}$ and $\A$ positive, hence $\A \in \mathscr{P}$ and $\X \in \Psi(\mathscr{P})$\,.

\noindent Conversely, if $\A \in \mathscr{P}$, then $\Psi(\A) \in \mathscr{S}$ and $\Psi(\A)\J_{2n} = \J_{2n}\overline{\Psi(\A)}$ by the above block form\,.

\end{proof}

\noindent We denote by $\tau$ (see Section \ref{SectionFour}) the Riemannian metric defined on $\mathscr{S}$ by
\begin{equation}
\tau_{\Q}(\X\,, \Y) = \tr(\Q^{-1}\X\Q^{-1}\Y)\,, \qquad \left(\X\,, \Y \in \T_{\A}(\mathscr{S}) \cong \mathfrak{s}\right)\,.
\label{FormTauS}
\end{equation}

\begin{lemma}

\label{LemmaBeforeProof}

The submanifold $\Psi(\mathscr{P})$ is stable under inversion, congruence, and real powers:
\begin{enumerate}
\item If $\X \in \Psi(\mathscr{P})$, then $\X^{-1} \in \Psi(\mathscr{P})$\,.
\item If $\X \in \Psi(\mathscr{P})$, then $\X^{\frac{1}{2}} \in \Psi(\mathscr{P})$, hence $\X^{t} \in \Psi(\mathscr{P})$ for all $t \in \mathbb{R}$\,.
\item If $\A\,, \B \in \Psi(\mathscr{P})$, then $\A^{\frac{1}{2}}\left(\A^{-\frac{1}{2}}\B\A^{-\frac{1}{2}}\right)^{t}\A^{\frac{1}{2}} \in \Psi(\mathscr{P})$ for all $t \in \mathbb{R}$\,.
\item If $\A\,, \B \in \Psi(\mathscr{P})$, the geodesic in $\mathscr{S}$ (with respect to $\tau$) joining $\A$ and $\B$ stays in $\Psi(\mathscr{P})$ for all $t \in \left[0\,, 1\right]$\,.
\end{enumerate}

\end{lemma}

\begin{proof}

\begin{enumerate}
\item If $\X\J = \J\overline{\X}$, then multiplying on the left by $\X^{-1}$ gives $\J = \X^{-1}\J\overline{\X}$, and multiplying on the right by $\overline{\X}^{-1} =\overline{\X^{-1}}$ yields $\X^{-1}\J = \J\overline{\X^{-1}}$, so $\X^{-1} \in \Psi(\mathscr{P})$\,.
\item Let $\X \in \Psi(\mathscr{P})$. Then $\X \in \mathscr{S}$ is positive definite Hermitian and $\X\J_{2n} = \J_{2n}\overline{\X}$. Let $\Y := \X^{\frac{1}{2}}$ be the unique positive definite Hermitian square root of $\X$ in $\mathscr{S}$. Taking complex conjugates in $\Y^{2} = \X$ gives $\overline{\Y}^{2} = \overline{\X}$, hence
\begin{equation*}
\left(\J_{2n}\overline{\Y}\J_{2n}^{-1}\right)^{2} = \J_{2n}\overline{\Y}^{2}\J_{2n}^{-1} =\J_{2n}\overline{\X}\J_{2n}^{-1}\,.
\end{equation*}
Using that $\X\J_{2n} = \J_{2n}\overline{\X}$ and $\J_{2n}^{-1} = -\J_{2n}$, we have
$\J_{2n}\overline{\X}\J_{2n}^{-1}=\X$, so
\begin{equation*}
\left(\J_{2n}\overline{\Y}\J_{2n}^{-1}\right)^{2} = \X\,.
\end{equation*}
Moreover, $\J_{2n}\overline{\Y}\J_{2n}^{-1}$ is again Hermitian positive definite (since $\Y$ is).
Thus $\J_{2n}\overline{\Y}\J_{2n}^{-1}$ is a positive definite Hermitian square root of $\X$.
By uniqueness of the positive definite Hermitian square root, we obtain
\begin{equation*}
\J_{2n}\overline{\Y}\J_{2n}^{-1} = \Y\,,
\end{equation*}
which is equivalent to $\Y\J_{2n} = \J_{2n}\overline{\Y}$. Therefore $\X^{\frac{1}{2}} = \Y \in \Psi(\mathscr{P})$. Finally, for any $t\in\mathbb{R}$, define $\X^{t}$ by functional calculus in $\mathscr{S}$.
Since the relation $\M\J_{2n} = \J_{2n}\overline{\M}$ is preserved under holomorphic (hence continuous) functional calculus on the spectrum of $\X$, it follows that $\X^{t}\J_{2n} =\J_{2n}\overline{\X^{t}}$, i.e. $\X^{t} \in \Psi(\mathscr{P})$ for all $t \in \mathbb{R}$.
\item If $\A\,, \B \in \Psi(\mathscr{P})$, then by (1) and (2) we have $\A^{\pm \frac{1}{2}} \in \Psi(\mathscr{P})$, and the product of two elements satisfying $\X\J = \J\overline{\X}$ again satisfies this relation: indeed, if $\X\J = \J\overline{\X}$ and $\Y\J = \J\overline{\Y}$, then
\begin{equation*}
\left(\X\Y\right)\J = \X\left(\Y\J\right) = \X(\J\overline{\Y}) = \left(\X\J\right)\overline{\Y} = \J\overline{\X}\overline{\Y} = \J\overline{\X\Y}\,.
\end{equation*}
Therefore $\M := \A^{-\frac{1}{2}}\B\A^{-\frac{1}{2}} \in \Psi(\mathscr{P})$ and hence $\M^{t} \in \Psi(\mathscr{P})$ for all $t$ by (2). Multiplying again by $\A^{\frac{1}{2}}$ on both sides shows
$\A^{\frac{1}{2}}\M^{t}\A^{\frac{1}{2}} \in \Psi(\mathscr{P})$ for all $t \in \mathbb{R}$\,.
\item Using \cite[Chapter~6]{BHATIA}, the geodesic $\widetilde{\gamma}$ between $\A$ and $\B$ in $\mathscr{S}$ is given by $\widetilde{\gamma}(t) = \A^{\frac{1}{2}}\left(\A^{-\frac{1}{2}}\B\A^{-\frac{1}{2}}\right)^{t}\A^{\frac{1}{2}}$ is the equation of the geodesics between $\A$ and $\B$ in $\mathscr{S}$ with respect to $\tau$, and using (3), we have that $\gamma(t) \in \Psi(\mathscr{P})$ for all $t \in \left[0\,, 1\right]$\,. 
\end{enumerate}

\end{proof}

\begin{remark}

Using Equations \eqref{FormOmegaP} and \eqref{FormTauS}, it follows that the metric $\omega$ on $\mathscr{P}$ is the pull-back of the metric $\tau$ on $\Psi(\mathscr{P}) \subseteq \mathscr{S}$\,. 

\label{RemarkPullBack}

\end{remark}

\begin{proof}[Proof of Theorem \ref{MainTheorem}]

The map $\Psi$ defined in Equation \eqref{EquationPsiMatrices}
\begin{equation*}
\Psi: \left(\mathscr{P}\,, \omega\right) \longrightarrow \left(\Psi(\mathscr{P})\,, \tau_{|_{\Psi(\mathscr{P})}}\right)
\end{equation*}
is an isometry. Let $\A\,, \B \in \mathscr{P}$ and set $\widetilde{\A} = \Psi(\A)$, $\widetilde{\B} = \Psi(\B)$ in $\Psi(\mathscr{P})$. As explained in Lemma \ref{LemmaBeforeProof}, the equation of the geodesic in $\Psi(\mathscr{P})$ between $\widetilde{\A}$ and $\widetilde{\B}$ is given by
\begin{equation*}
\widetilde{\gamma}(t) = \widetilde{\A}^{\frac{1}{2}}\left(\widetilde{\A}^{-\frac{1}{2}}\widetilde{\B}\widetilde{\A}^{-\frac{1}{2}}\right)^{t}\widetilde{\A}^{\frac{1}{2}}\,, \qquad \left(t \in \left[0\,, 1\right]\right)\,.
\end{equation*}
Using that the map $\Psi$ is an isometry, it follows from \cite[Proposition~5.4]{LEE} that the equation of the geodesic in $\mathscr{P}$ between $\A$ and $\B$ is given by
\begin{equation*}
\gamma(t) = \Psi^{-1}\left(\widetilde{\gamma}(t)\right)\,, \qquad \left(t \in \left[0\,, 1\right]\right)\,,
\end{equation*}
and using that $\Psi^{-1}(\Psi(\M)^{t}) = \M^{t}$ for all $\M \in \mathscr{P}$ and $t \in \left[0\,, 1\right]$, we get that 
\begin{equation*}
\gamma(t) = \A^{\frac{1}{2}}\left(\A^{-\frac{1}{2}}\B\A^{-\frac{1}{2}}\right)^{t}\A^{\frac{1}{2}}\,.
\end{equation*}

\end{proof}

\subsection{Geometric mean on $\mathscr{P}$}

For two matrices $\A\,, \B \in \mathscr{P}$ and $t \in \left[0\,, 1\right]$, we denote by $\A \sharp_{t} \B$ the matrix in $\Mat(n)$ given by
\begin{equation*}
\A \sharp_{t} \B = \A^{\frac{1}{2}}\left(\A^{-\frac{1}{2}}\B\A^{-\frac{1}{2}}\right)^{t}\A^{\frac{1}{2}}\,,
\end{equation*}
and let $\A \sharp \B = \A \sharp_{\frac{1}{2}} \B$. As explained in Section \ref{SectionGeodesicsOnP}, $\A \sharp_{t} \B \in \mathscr{P}$ for all $t \in \left[0\,, 1\right]$\,.

\noindent We now collect several basic properties of the geometric mean $\sharp_{t}$ on $\mathscr{P}$. Since the proofs are identical to those in the real and complex settings, they are omitted.

\begin{proposition}

\begin{itemize}
\item For all $\A\,, \B \in \mathscr{P}$ and $t \in \left(0\,, 1\right)$, $\A \sharp_{t} \A = \A$ and $\A \sharp_{t} \B = \A$ if and only if $\A = \B$\,.
\item For all $\A\,, \B \in \mathscr{P}$, $\A \sharp \B$ is the unique solution of Riccati's equation
\begin{equation*}
\X\A^{-1}\X = \B\,,
\end{equation*}
\item For all $\A\,, \B \in \mathscr{P}$ and $\C \in \GL(n)$, we have
\begin{equation*}
\left(\C^{*}\A\C\right) \sharp \left(\C^{*}\B\C\right) = \C^{*}\left(\A \sharp \B\right)\C\,.
\end{equation*}
\item For all $\A_{1}\,, \A_{2}\,, \B_{1}\,, \B_{2} \in \mathscr{P}$ such that $\A_{1} \preceq \A_{2}$ and $\B_{1} \preceq \B_{2}$, we get
\begin{equation*}
\A_{1} \sharp \B_{1} \preceq \A_{2} \sharp \B_{2}\,.
\end{equation*}
\item For all $\A\,, \B \in \mathscr{P}$, then $\left(\A \sharp \B\right)^{-1} = \A^{-1} \sharp \B^{-1}$\,.
\end{itemize}

\end{proposition}

\noindent We denote by $\preceq$ the Loewner order on $\mathscr{P}$ as in Section \ref{SectionThree}\,.

\begin{proposition}

For all $\A\,, \B \in \mathscr{P}$, we obtain
\begin{enumerate}
\item $\A \sharp \B = \B \sharp \A$\,,
\item $\left(\A \sharp \B\right)^{-1} = \A^{-1} \sharp \B^{-1}$\,.
\end{enumerate}
Moreover, let $t \in \left(0\,, 1\right)$\,.
\begin{enumerate}
\item For all $\A \in \mathscr{P}$, we have $\A \sharp_{t} \A = \A$. Moreover, we have $\A \sharp_{t} \B = \A$ if and only if $\A = \B$\,.
\item For all $\A\,, \B \in \mathscr{P}$ and $\alpha, \beta >0$, we have
\begin{equation*}
\left(\alpha\A\right) \sharp_{t} \left(\beta\B\right) = \alpha^{1-t}\beta^{t} \A \sharp_{t} \B\,.
\end{equation*}
\item For all $\A\,, \B \in \mathscr{P}$, we have
\begin{equation*}
\A \sharp_{t} \B = \B \sharp_{1-t} \A\,.
\end{equation*}
\item If $\A\,, \B\,, \C\,, \D \in \mathscr{P}$ are such that $\A \preceq \C$ and $\B \preceq \D$, then
\begin{equation*}
\A \sharp_{t} \B \preceq \C \sharp_{t} \D\,.
\end{equation*}
\end{enumerate}
\end{proposition}

\noindent We finish this appendix with the following proposition\,.

\begin{proposition}

\begin{enumerate}
\item For all $g \in \Sp(n)$, $t\in \left[0\,, 1\right]$ and $\A\,, \B \in \mathscr{P}$, we have
\begin{equation*}
\left(g\A g^{\sharp}\right) \sharp_{t} \left(g\B g^{\sharp}\right) = g\left(\A \sharp_{t} \B\right)g^{\sharp}\,.
\end{equation*}
\item For all $\A\,, \B\,, \C\,, \D \in \mathscr{P}$ and $t\,, s \in \left[0\,, 1\right]$, we have
\begin{equation*}
\left(1-s\right)\A \sharp_{t} \C + s \B \sharp_{t} \D \preceq \left(\left(1-s\right)\A + s \B\right) \sharp_{t} \left(\left(1-s\right)\C + s \D\right)\,.
\end{equation*}
\item For all $\A\,, \B \in \mathscr{P}$ and $t\,, s\,, u \in \left[0\,, 1\right]$,
\begin{equation*}
\left(\A \sharp_{t} \B\right) \sharp_{u} \left(\A \sharp_{s} \B\right) = \A \sharp_{(1-u)t+us}\B\,.
\end{equation*}
\end{enumerate}

\end{proposition}


\begin{thebibliography}{10}

\bibitem{AndoHiai1994}
Tsuyoshi Ando and Fumio Hiai.
\newblock Log majorization and complementary {G}olden-{T}hompson type
  inequalities.
\newblock volume 197/198, pages 113--131. 1994.
\newblock Second Conference of the International Linear Algebra Society (ILAS)
  (Lisbon, 1992).

\bibitem{BebianoLemosProvJrSoares2012}
N.~Bebiano, R.~Lemos, J.~da~Provid\^{e}ncia, and G.~Soares.
\newblock Operator inequalities for {$J$}-contractions.
\newblock {\em Math. Inequal. Appl.}, 15(4):883--897, 2012.

\bibitem{BebianoLemosSoares2024a}
N.~Bebiano, R.~Lemos, and G.~Soares.
\newblock {$J$}-selfadjoint matrix means and their indefinite inequalities.
\newblock {\em Acta Sci. Math. (Szeged)}, 90(3-4):513--525, 2024.

\bibitem{BebianoLemosSoares2023}
N.~Bebiano, R.~Lemos, and G.~Soares.
\newblock On the hyperbolicity of the {K}rein space numerical range.
\newblock {\em Linear Multilinear Algebra}, 72(14):2267--2287, 2024.

\bibitem{BebianoLemosSoares2025}
N.~Bebiano, R.~Lemos, and G.~Soares.
\newblock Matrices with hyperbolical {K}rein space numerical range.
\newblock {\em Adv. Oper. Theory}, 10(1):Paper No. 13, 19, 2025.

\bibitem{BHATIA}
Rajendra Bhatia.
\newblock {\em Positive definite matrices}.
\newblock Princeton Series in Applied Mathematics. Princeton University Press,
  Princeton, NJ, 2007.
\newblock [2015] paperback edition of the 2007 original [MR2284176].

\bibitem{BHATIAH}
Rajendra Bhatia and John Holbrook.
\newblock Riemannian geometry and matrix geometric means.
\newblock {\em Linear Algebra Appl.}, 413(2-3):594--618, 2006.

\bibitem{ChoiKimLim2022}
Hayoung Choi, Sejong Kim, and Yongdo Lim.
\newblock A binomial expansion formula for weighted geometric means of
  unipotent matrices.
\newblock {\em Linear Multilinear Algebra}, 72(4):615--630, 2024.

\bibitem{Dehghani}
M.~Dehghani and S.~M.~S. Modarres~Mosadegh.
\newblock Operator arithmetic-harmonic mean inequality on {K}rein spaces.
\newblock {\em J. Math. Ext.}, 8(1):59--68, 2014.

\bibitem{FujiiKamei2006}
Masatoshi Fujii and Eizaburo Kamei.
\newblock Ando-{H}iai inequality and {F}uruta inequality.
\newblock {\em Linear Algebra Appl.}, 416(2-3):541--545, 2006.

\bibitem{HELGASON}
Sigurdur Helgason.
\newblock {\em Differential geometry, {L}ie groups, and symmetric spaces},
  volume~34 of {\em Graduate Studies in Mathematics}.
\newblock American Mathematical Society, Providence, RI, 2001.
\newblock Corrected reprint of the 1978 original.

\bibitem{HiaiKosaki2021}
Fumio Hiai and Hideki Kosaki.
\newblock Connections of unbounded operators and some related topics: von
  {N}eumann algebra case.
\newblock {\em Internat. J. Math.}, 32(5):Paper No. 2150024, 88, 2021.

\bibitem{HORN}
Roger~A. Horn and Charles~R. Johnson.
\newblock {\em Matrix analysis}.
\newblock Cambridge University Press, Cambridge, second edition, 2013.

\bibitem{KNAPP}
Anthony~W. Knapp.
\newblock {\em Lie groups beyond an introduction}, volume 140 of {\em Progress
  in Mathematics}.
\newblock Birkh\"{a}user Boston, Inc., Boston, MA, second edition, 2002.

\bibitem{KuboAndo1980}
Fumio Kubo and Tsuyoshi Ando.
\newblock Means of positive linear operators.
\newblock {\em Math. Ann.}, 246(3):205--224, 1979/80.

\bibitem{LawsonLim2014}
Jimmie Lawson and Yongdo Lim.
\newblock Weighted means and {K}archer equations of positive operators.
\newblock {\em Proc. Natl. Acad. Sci. USA}, 110(39):15626--15632, 2013.

\bibitem{LEE}
John~M. Lee.
\newblock {\em Introduction to {R}iemannian manifolds}, volume 176 of {\em
  Graduate Texts in Mathematics}.
\newblock Springer, Cham, 2018.
\newblock Second edition of [MR1468735].

\bibitem{TRIO}
Ming Liao, Xuhua Liu, and Tin-Yau Tam.
\newblock A geometric mean for symmetric spaces of noncompact type.
\newblock {\em J. Lie Theory}, 24(3):725--736, 2014.

\bibitem{Pusz}
W.~Pusz and S.~L. Woronowicz.
\newblock Functional calculus for sesquilinear forms and the purification map.
\newblock {\em Rep. Mathematical Phys.}, 8(2):159--170, 1975.

\bibitem{SCHARLAU}
Winfried Scharlau.
\newblock {\em Quadratic and {H}ermitian forms}, volume 270 of {\em Grundlehren
  der mathematischen Wissenschaften [Fundamental Principles of Mathematical
  Sciences]}.
\newblock Springer-Verlag, Berlin, 1985.

\bibitem{ZHANG}
Fuzhen Zhang.
\newblock Quaternions and matrices of quaternions.
\newblock {\em Linear Algebra Appl.}, 251:21--57, 1997.

\end{thebibliography}
\end{document}